\documentclass[11pt, a4paper]{amsart}
\usepackage{amssymb, array, amsmath, amscd, pdfpages, enumerate, amsthm, setspace, hyperref,mathtools}

\usepackage[utf8]{inputenc}
\usepackage{amssymb}
\usepackage{bm}
\usepackage{tikz}
\usetikzlibrary{shapes,arrows}
\usepackage{graphicx} 
\usepackage{comment}
\usepackage{subfig}

\DeclareMathOperator{\re}{Re}
\DeclareMathOperator{\im}{Im}

\newcommand*{\R}{{\mathbb{R}}}

\newcommand*{\abs} [1]{\lvert#1\rvert}
\newcommand*{\norm}[1]{\lVert#1\rVert}

\newcommand*{\Abs}[2][default]{\ifthenelse{\equal{#1}{default}}{\left\lvert#2\right\rvert}{\ldelim{#1}{\lvert}#2\rdelim{#1}{\rvert}}}
\newcommand*{\Norm}[2][default]{\ifthenelse{\equal{#1}{default}}{\left\lVert#2\right\rVert}{\ldelim{#1}{\lVert}#2\rdelim{#1}{\rVert}}}

\newcommand*{\Iprod}[3][default]{\ifthenelse{\equal{#1}{default}}{\left\langle#2,#3\right\rangle}{\ldelim{#1}{\langle}#2,#3\rdelim{#1}{\rangle}}}
\newcommand*{\Dualpair}[3][default]{\ifthenelse{\equal{#1}{default}}{\left\langle#2,#3\right\rangle}{\ldelim{#1}{\langle}#2,#3\rdelim{#1}{\rangle}}}

\newcommand{\eq}[1]{\begin{align*}#1\end{align*}}
\newcommand{\eqn}[1]{\begin{align}#1\end{align}}

\newcommand{\ga}{\alpha}

\newcommand{\gw}{\omega}

\newcommand*{\ddb}[2][1]{\ifthenelse{\equal{#1}{1}}{\frac{d}{d#2}}{\frac{d^{#1}}{d#2^{#1}}}}
\newcommand*{\pd}[3][1]{\ifthenelse{\equal{#1}{1}}{\frac{\partial{#2}}{\partial{#3}}}{\frac{\partial^{#1}{#2}}{\partial#3^{#1}}}}

\newtheorem{theorem}{Theorem}[section]
\newtheorem{lemma}[theorem]{Lemma}

\theoremstyle{definition}
\newtheorem{definition}[theorem]{Definition}

\newtheorem{remark}[theorem]{Remark}

\title[Robust Controllers for a Flexible Satellite Model]
{Robust Controllers for a Flexible Satellite Model} 

\author[T. Govindaraj]{Thavamani Govindaraj}
\address[T. Govindaraj]{Mathematics, Faculty of Information Technology and Communication Sciences, Tampere University, PO.\ Box 692, 33101 Tampere, Finland}
\email{thavamani.govindaraj@tuni.fi}

\author[J.-P. Humaloja]{Jukka-Pekka Humaloja}
\address[J.-P. Humaloja]{University of Alberta, 9211-116 St, Edmonton, AB T6G 1H9, Canada}
\email{jphumaloja@ualberta.ca}

\author[L. Paunonen]{Lassi Paunonen}
\address[L. Paunonen]{Mathematics, Faculty of Information Technology and Communication Sciences, Tampere University, PO.\ Box 692, 33101 Tampere, Finland}
\email{lassi.paunonen@tuni.fi}

\subjclass{93C20, 93D23, 93B52, 35L99 (93B28).}
 \keywords{Distributed parameter systems, stability analysis, output regulation, coupled PDE-ODE system, boundary control system.}

\thanks{The research was supported by the Academy of Finland grants 298182, 310489 and 349002.
J.-P. Humaloja is funded by grants from the Jenny and Antti Wihuri Foundation and the Vilho, Yrj\"o and Kalle V\"ais\"al\"a Foundation}

\begin{document}
\maketitle

\begin{abstract}
We consider a PDE-ODE model of a flexible satellite that is composed of two identical flexible solar panels and a center rigid body. We prove that the satellite model is exponentially stable in the sense that the energy of the solutions decays to zero exponentially.  
In addition, we construct two internal model based controllers, a passive controller and an observer based controller, such that the linear and angular velocities of the center rigid body converge to the given sinusoidal signals asymptotically. 
A numerical simulation is presented to compare the performances of the two controllers.
\end{abstract}


\section{Introduction}

In this paper, we consider output tracking and disturbance rejection problem for a flexible satellite that is composed of two identical flexible solar panels and a center rigid body (Figure \ref{satellite model}). 
Modeling the satellite panels as viscously damped Euler-Bernoulli beams of length 1, the satellite system we study is given by (similar models can be found in \cite{Bontsema1988}, \cite{WeiHe2015})

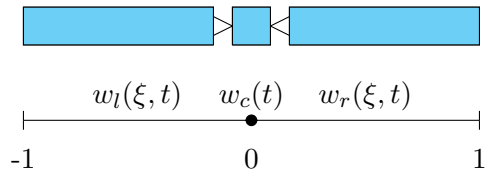
\begin{figure}[!ht]
\centering
\begin{tikzpicture}
  \draw (3,0) -- (0.5,0) -- (0.5,0.5) -- (3,0.5) -- (3,0);
	\draw (3,0.125) -- (3.25,0.25);
	\draw (3,0.375) -- (3.25,0.25);
	\draw (3.25,0) -- (3.25,0.5) -- (3.75,0.5) -- (3.75,0) -- (3.25,0);
	\draw (3.75,0.25) -- (4,0.125);
	\draw (3.75,0.25) -- (4,0.375);
	\draw (4,0) -- (4,0.5) -- (6.5,0.5) -- (6.5,0) -- (4,0);
	\filldraw[fill = cyan!50!] (0.5,0) rectangle (3,0.5);
	\filldraw[fill = cyan!50!] (3.25,0) rectangle (3.75,0.5);
	\filldraw[fill = cyan!50!] (4,0) rectangle (6.5,0.5);
	\filldraw
(0.5,-1) -- (3.5,-1) node[pos = 0.5, sloped, above] {$w_l(\xi,t)$}  circle (2pt) node[sloped,above] {$w_c(t)$}     -- 
(6.5,-1)  node[pos = 0.5,sloped,above] {$w_r(\xi,t)$};
\draw (0.5,-1.125) -- (0.5,-1);
\draw (0.5,-0.875) -- (0.5,-1);
\draw (6.5,-1.125) -- (6.5,-1);
\draw (6.5,-0.875) -- (6.5,-1);
\draw (0.5,-1.25) node[align=left,below]{-1};
\draw (3.5,-1.25) node[align=center,below]{0};
\draw (6.5,-1.25) node[align=right,below]{1};
\end{tikzpicture}
\caption{Satellite with flexible solar panels}
\label{satellite model}
\end{figure}	

\begin{equation} \label{eq:pdeode}
\begin{gathered}
\rho a\ddot{w}_l(\xi,t)+EI w_l^{\prime\prime\prime\prime}(\xi,t)+\gamma \dot{w}_l(\xi,t)=b_{d1}(\xi)w_{d1}(t),\  -1<\xi<0, t>0,\\
\rho a\ddot{w}_r(\xi,t)+EI w_r^{\prime\prime\prime\prime}(\xi,t)+\gamma \dot{w}_r(\xi,t)=b_{d2}(\xi)w_{d2}(t),\ 0<\xi<1, t>0,\\
m\ddot{w}_c(t)=EI w_l^{\prime\prime\prime}(0,t)-EI w_r^{\prime\prime\prime}(0,t)+u_1(t)+w_{d3}(t),\\
I_m\ddot{\theta}_c(t)=-EI w_l^{\prime\prime}(0,t)+EI w_r^{\prime\prime}(0,t)+u_2(t)+w_{d4}(t),\\
w_l^{\prime\prime}(-1,t)=0,\quad \ w_l^{\prime\prime\prime}(-1,t)=0,\\
w_r^{\prime\prime}(1,t)=0,\quad \ w_r^{\prime\prime\prime}(1,t)=0,\\
\dot{w}_l(0,t)=\dot{w}_r(0,t)=\dot{w_c}(t),\\
\dot{w}_l^{\prime}(0,t)=\dot{w}_r^{\prime}(0,t)=\dot{\theta}_c(t),
\end{gathered}
\end{equation}
where $w_l(\xi,t)$ and $w_r(\xi,t)$ are the transverse displacements of the left and the right beam, respectively, $\dot{w}_l(\xi,t)$ and $w_l'(\xi,t)$ denote time and spatial derivatives of $w_l(\xi,t)$, respectively, $w_c(t)$ and $\theta_c(t)$ are the linear and angular displacements of the rigid body, respectively, $u_1(t)$ and $u_2(t)$ are external control inputs of the satellite model, $w_{d1}(t)$, $w_{d2}(t)$, $w_{d3}(t)$ and $w_{d4}(t)$ are external disturbances in the satellite model, $b_{d1}(\cdot) \in L^2(-1,0)$ and $b_{d2}(\cdot) \in L^2(0,1)$ are real-valued functions. Here $\dot{w}_c(t)=\dot{w}_l(\xi,t)|_{\xi=0}=\dot{w}_r(\xi,t)|_{\xi=0}$ and $\dot{\theta}_c(t)=\dot{w}_l^{\prime}(\xi,t)|_{\xi=0}=\dot{w}_r^{\prime}(\xi,t)|_{\xi=0}$ are linear and angular velocities of the rigid body, respectively. The parameters $a$, $\rho$, $E$, $I$ and $\gamma$ are cross sectional area, linear density, Young's modulus of elasticity, second moment of area of the cross section and the viscous damping coefficient of the beams, respectively, and $m$ and $I_m$ denote the mass and the mass moment of inertia of the center rigid body.
Measurements that are the outputs of the model are taken on the center rigid body and are given by,
\begin{equation}\label{outputs}
y_1(t)=\dot{w}_c(t),\ y_2(t)=\dot{\theta}_c(t).
\end{equation}

The main control objective is to construct a dynamic error feedback controller such that the outputs, the linear and the angular velocities of the center rigid body, track given reference signals $y_{ref}(t)$ asymptotically. i.e.,
\begin{align*}
\norm{y(t)-y_{ref}(t)} \rightarrow 0 \quad \text{as} \quad t \rightarrow \infty,
\end{align*}
where $y(t)=(y_1(t),y_2(t))^T$ is the output of the satellite model. In addition, the proposed controller is required to be robust in the sense that it achieves output tracking despite perturbations, disturbances and uncertainties in the satellite system.

As the first main contribution of this paper, we present a detailed proof of uniform exponential stability of the satellite model in the sense that the energy of the solutions decay exponentially to zero. 
The stability proof is based on the results from $C_0$-semigroup theory.  
We write the satellite system as a coupled system of a PDE (two beams are combined into a single system) and an ODE (rigid body) via a power preserving interconnection. The main proof is divided naturally into two steps. In the first step, we show that the imaginary axis is included in the resolvent set of the satellite system operator. In the second step, we derive an explicit expression for the resolvent operator 
and show that it is uniformly bounded on the imaginary axis. The stability proof is challenging because the input and the output operators of the PDE are not admissible and its transfer function is not well-posed (in the sense that the input-output map of the PDE is unbounded).

As the second main contribution of this paper, we construct two robust controllers, a passive controller \cite{RebarberWeiss2003}, \cite{Lassi2019} and an observer based controller \cite{HamPoh2010}, \cite{Pau2016}, for the robust output tracking of the satellite model. The proposed controller designs are based on the internal model principle \cite{Davison76},  \cite{Francis1975}, \cite{HamPoh2010}, \cite{Pau2016}, \cite{PauPoh2014}. 
Finally, simulation results testing the effectiveness of the controllers are presented.

There are several studies in the literature investigating control problems of satellite models. In \cite{Bontsema1988}, the stabilization problem of a flexible spacecraft has been investigated using frequency domain approach. In \cite{WeiHe2015}, dynamic modeling and vibration control of a flexible satellite has been considered and vibrations of the solar panels have been suppressed using the single-point control input on the center body. In \cite{Aoues2019},  modeling and control of a rotating flexible spacecraft has been considered, a Proportional Derivative controller and a nonlinear controller have been presented to suppress elastic vibrations of the satellite model.  References \cite{WeiHe2015} and \cite{Aoues2019} use Lyapunov methods to prove the stability of the models. To the best of our knowledge, robust output tracking problem for flexible satellites has not been considered in the literature. 

Stability of coupled PDE-ODE systems can often be obtained using controllability and observability results. In \cite{ZhaWei11}, controllability and observability results of a well-posed and strictly proper linear system coupled with a finite-dimensional linear system with an invertible first component in its feedthrough matrix were presented. In \cite{ZhaWei17}, using results from \cite{ZhaWei11}, strong stability of coupled impedance passive systems was shown and the results were applied to the SCOLE model to show that the SCOLE model coupled with tuned mass damper system is strongly stable. Moreover, the SCOLE model is not exactly controllable in the natural energy state space (\cite[Sec. 1]{ZhaoWeiss2010}) but it was shown in \cite{ZhaoWeiss2010} that the SCOLE model is exactly controllable in a smoother state space. In \cite{ZhaoWeiss2018}, it was shown that a coupled system consisting of a well-posed and impedance passive linear system and an internal model based controller in a feedback connection is strongly stable. In our case, since the beam system in the satellite model is not well-posed on the natural energy state space and the rigid center body has no feedthrough term, the results of \cite{ZhaWei11}, \cite{ZhaWei17} cannot be utilized in showing the exponential stability of the satellite system. Moreover, since our aim is to achieve exponential stability of the closed-loop system and one of the proposed controllers is infinite-dimensional, the results in \cite{ZhaoWeiss2010}, \cite{ZhaWei11}, \cite{ZhaWei17} and \cite{ZhaoWeiss2018} are not applicable in showing the exponential stability of the closed-loop system consisting of the satellite system and the controller. The results in the above mentioned references have unstable infinite-dimensional part and therefore only strong stability of the coupled system was obtained. In this work, since the beam system is exponentially stable due to the distributed damping, we are able to prove the exponential stability of the satellite system. 

A preliminary version of these results has been presented in IFAC World Congress 2020 \cite{GovHumPau2020}. As the main novelty of this version with respect to \cite{GovHumPau2020}, we present a detailed proof of the exponential stability of the satellite system. We present a passive controller and an observer based controller which also achieve the robust output tracking of the satellite model and reject external disturbances. In addition, simulation results showing the performances of the controllers are presented. 

The paper is organized as follows. In Section \ref{sec.absfor}, we present the abstract formulation of the satellite model. In Section \ref{sec.sta}, we present some technical lemmas and prove the exponential stability of the satellite model. In Section \ref{sec.OR}, we present the tracking problem, the reference signal to be tracked by the satellite model and the disturbance signals to be rejected. We present two internal model based controllers for the robust output tracking of the satellite model. In addition, simulation results are presented for particular choices of reference and disturbance signals.	In Section \ref{sec.con}, we conclude our results.

\subsection{Notation}
For normed linear spaces $X$ and $Y$, $\mathcal{L}(X,Y)$ denotes the set of bounded linear operators from $X$ to $Y$. For a linear operator $A$, $D(A), \mathcal{R}(A)$ and $\mathcal{N}(A)$ denote domain, range and kernel of $A$, respectively. The resolvent and the spectrum of $A$ are denoted by $\rho(A)$ and $\sigma(A)$, respectively. The resolvent operator is denoted by $R(\lambda,A)=(\lambda-A)^{-1},\ \lambda \in \rho(A)$. We denote by $X_{-1}$ the completion of $X$ with respect to the norm $\|x\|_{-1} = \|(\beta I-A)^{-1}x\|,\ x \in X,\ \beta\in \rho(A)$ and by $A_{-1}\in \mathcal{L}(X,X_{-1})$, the extension of $A$ to $X_{-1}$. For functions $f,g:I\subset \mathbb{R}\to \mathbb{R}_+$ and $f_k,g_k\geq 0$, we denote $f(x)\lesssim g(x)$ and $f_k\lesssim g_k$ if there exist $M_1, M_2>0$ such that $f(x)\leq M_1 g(x)$ and $f_k\leq M_2 g_k$ for all values of $x \in I$ and $k\in J\subset \mathbb{N}$.

\section{Abstract Formulation of the Satellite Model}\label{sec.absfor}
In this section, we write our satellite model (\ref{eq:pdeode})-(\ref{outputs}) in the state space form 
\begin{equation}\label{statespaceform}
\begin{aligned}
\dot{x}(t)&=Ax(t)+Bu(t)+B_dw_d(t),\quad x(0)=x_0,\\
y(t)&=Cx(t)
\end{aligned}
\end{equation}
where $x(t)\in X$ is the state variable and $X$ is a Hilbert space, $u(t) \in U=\mathbb{R}^2$ is the control input, $w_d(t) \in U_d=\mathbb{R}^4$ is the external disturbance and $y(t) \in Y=\mathbb{R}^2$ is the output. The operator $A:D(A)\subset X \to X$ generates a strongly continuous semigroup on $X$ and the operators $B \in \mathcal{L}(U,X)$, $B_d \in \mathcal{L}(U_d,X)$ and $C\in \mathcal{L}(X,Y)$ are bounded. The formulation (\ref{statespaceform}) will be used in Section 4 in the construction of controllers for robust output regulation.

In order to write the satellite model (\ref{eq:pdeode})-(\ref{outputs}) in the state space form, we decompose the satellite system into a PDE system (the two beams combined into a single system) coupled with an ODE system (center rigid body) where PDE interacts with ODE via boundary controls and boundary observations called ``virtual boundary inputs'' and ``virtual boundary outputs", respectively. Figure \ref{interconnection} shows the boundary interconnections between the beams and the center rigid body. This type of decomposition has been considered, for example, in \cite{ZhaoWeiss2010} for SCOLE model. 
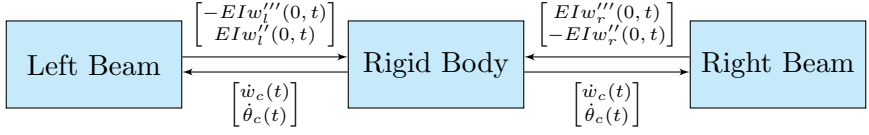
\begin{figure}[!ht]
\centering
\begin{tikzpicture}[auto, node distance=1.5cm,>=latex']
		\tikzstyle{block} = [draw, fill=cyan!20, rectangle, 
    minimum height=3em, minimum width=6em]
\tikzstyle{sum} = [draw, fill=blue!20, circle, node distance=1cm]
\tikzstyle{input1} = [coordinate, node distance=1cm]
\tikzstyle{linput} = [coordinate]
\tikzstyle{clinput} = [coordinate]
\tikzstyle{crinput} = [coordinate]
\tikzstyle{routput} = [coordinate]
\tikzstyle{output1} = [coordinate, node distance=1cm]
\tikzstyle{pinstyle} = [pin edge={to-,thin,black}]

    \node [input1, name=input1] {};
		\node [block, right of=input1] (Left Beam) {Left Beam};
		\node [linput, right of=Left Beam] (linput) {};
		\node [clinput, right of=linput] (clinput) {};
		\node [block, right of=clinput] (Rigid Body) {Rigid Body};
		\node [crinput, right of=Rigid Body] (crinput) {};
		\node [routput, right of=crinput] (routput) {};
		\node [block, right of=routput] (Right Beam) {Right Beam};
		\node [output1, right of=Right Beam] (output1) {};
		\draw [->] ([yshift=0.1cm] Left Beam.east) -- node[pos=0.5] {\tiny$\begin{bmatrix}-EI w_l^{\prime\prime\prime}(0,t)\\ EI w_l^{\prime\prime}(0,t)\end{bmatrix}$} ([yshift=0.1cm] Rigid Body.west);
		\draw [->] ([yshift=-0.1cm] Rigid Body.west) -- node {\tiny $\begin{bmatrix} \dot{w}_c(t) \\ \dot{\theta}_c(t) \end{bmatrix}$} ([yshift=-0.1cm] Left Beam.east);
		\draw [->] ([yshift=0.1cm] Right Beam.west) -- node [above] {\tiny$\begin{bmatrix}EI w_r^{\prime\prime\prime}(0,t) \\ -EI w_r^{\prime\prime}(0,t)\end{bmatrix}$} ([yshift=0.1cm] Rigid Body.east);
		\draw [->] ([yshift=-0.1cm] Rigid Body.east) -- node[below] {\tiny$\begin{bmatrix} \dot{w}_c(t) \\ \dot{\theta}_c(t) \end{bmatrix}$} ([yshift=-0.1cm] Right Beam.west);
\end{tikzpicture}
\caption{Coupling of the beams with the rigid body}
\label{interconnection}
\end{figure}
As the first step towards state space formulation, we write the PDE as an impedance passive abstract boundary control and observation system given by the following definitions.

\begin{definition}[Boundary Control and Observation System {\cite[Def. 3.3.2]{CurZwa1995}, \cite[Ch. 11]{JacZwa2011}}]
Let $\hat{X}$, $\hat{U}$ and $\hat{Y}$ be Hilbert spaces. Consider the system
\begin{subequations}\label{BCS}
\begin{align}
\dot{\hat{x}}(t)&=\mathcal{\hat{A}}\hat{x}(t),\quad \hat{x}(0)=\hat{x}_{0}, \label{bcs1}\\
\mathcal{\hat{B}}\hat{x}(t)&=\hat{u}(t)\label{bcs2},\\
\hat{y}(t)&=\mathcal{\hat{C}}\hat{x}(t)\label{bcs3}
\end{align}
\end{subequations}
where $\mathcal{\hat{A}}:D(\mathcal{\hat{A}})\subset \hat{X}\rightarrow \hat{X}$,  $\mathcal{\hat{B}}:D(\mathcal{\hat{B}})\subset \hat{X} \rightarrow \hat{U}$ and $\mathcal{\hat{C}}:D(\mathcal{\hat{A}})\rightarrow \hat{Y}$ are linear operators and $D(\mathcal{\hat{A}})\subset D(\mathcal{\hat{B}})$. Then (\ref{BCS}) is a boundary control and observation system if the following hold.
\begin{enumerate}
	\item[1.] The operator $\hat{A}:D(\hat{A})\rightarrow \hat{X}$ with $D(\hat{A})=D(\mathcal{\hat{A}})\cap \mathcal{N}(\mathcal{\hat{B}})$ and $\hat{A}\hat{x}=\mathcal{\hat{A}}\hat{x}$ for $\hat{x}\in D(\hat{A})$ is the infinitesimal generator of a $C_0$-semigroup $(\hat{T}(t))_{t\geq 0}$ on $\hat{X}$.
	\item[2.] There exists an operator $\hat{H} \in \mathcal{L}(\hat{U},\hat{X})$ such that for all $\hat{u} \in \hat{U}$ we have $\hat{H}\hat{u} \in D(\mathcal{\hat{A}})$, $\mathcal{\hat{A}}\hat{H} \in \mathcal{L}(\hat{U},\hat{X})$ and $\mathcal{\hat{B}}\hat{H}\hat{u}=\hat{u},\ \hat{u}\in \hat{U}$.
\end{enumerate}
\end{definition}

\begin{remark} \label{unboundedB}
Let $(\mathcal{\hat{A}},\mathcal{\hat{B}})$ be a boundary control system. Then according to \cite[Ch. 10]{TucWei2009}, there exists a unique $\hat{B}\in \mathcal{L}(\hat{U},\hat{X}_{-1})$ such that $\mathcal{\hat{A}}=\hat{A}_{-1}+\hat{B}\mathcal{\hat{B}}$ on $D(\mathcal{\hat{A}})$ and therefore (\ref{bcs1}) and (\ref{bcs2}) can be written as 
\begin{align*}
\dot{\hat{x}}(t)&=\hat{A}_{-1}\hat{x}(t)+\hat{B}\hat{u}(t),\quad \hat{x}(0)=\hat{x}_0.
\end{align*}
\end{remark}
\begin{definition}[Impedance Passive System]
The system $(\mathcal{\hat{A}}, \mathcal{\hat{B}}, \mathcal{\hat{C}})$ is called imped\-ance passive if the solutions of \eqref{BCS} satisfy
\begin{align*}
\frac{1}{2}\frac{d}{dt}\norm{\hat{x}(t)}_{\hat{X}}^2 \leq \text{Re}\left\langle \hat{u}(t),\hat{y}(t) \right\rangle_{\hat{U},\hat{Y}}, \quad t>0.
\end{align*}
\end{definition}
We note that the above definition holds also for the systems in the state space form. 
Since we are interested in controlling velocities of the center rigid body, we use energy state space \cite{JacZwa2011} instead of natural state space in order to write the PDE as an abstract system.

\subsection{Abstract Formulation of the Beams}

The left beam system that we extract from the the satellite system is described by,
\begin{subequations}\label{leftbeam}
\begin{align}
&\ddot{w}_l(\xi,t)+\frac{EI}{\rho a} w_l^{\prime\prime\prime\prime}(\xi,t)+\frac{\gamma}{\rho a}\dot{w}_l(\xi,t)=0,\label{lb1}\\
&\dot{w}_l(0,t)=u_{l1}(t),\  \dot{w}_l^{\prime}(0,t)=u_{l2}(t),\label{lb2}\\
&w_l^{\prime\prime}(-1,t)=0,\quad \ w_l^{\prime\prime\prime}(-1,t)=0,\label{lb3}\\
&y_{l1}(t)=-EI w_l^{\prime\prime\prime}(0,t),\  y_{l2}(t)=EI w_l^{\prime\prime}(0,t).\label{lb4}
\end{align}
\end{subequations}
where $-1<\xi<0, t>0$ and $u_{l1}(t),\ u_{l2}(t)$ are the virtual boundary inputs and $y_{l1}(t),\ y_{l2}(t)$ are the virtual boundary outputs of the left beam (see Figure \ref{interconnection}), respectively.

By choosing the state variable 
\begin{align*}x_l(t)=\begin{bmatrix}\rho a \dot{w}_l(\cdot,t) \\ w_l^{\prime\prime}(\cdot,t)\end{bmatrix},\end{align*} where $\dot{w}_l(\cdot,t)$ and $w_l^{\prime\prime}(\cdot,t)$ are the velocity and the bending moment of the left beam, respectively, (\ref{leftbeam}) can be written in boundary control and observation form on the state space $X_l=L^2([-1,0];\mathbb{R}^2)$ as 
\begin{subequations}\label{eq:beam1}
\begin{align}
\dot{x}_l(t)&=\mathcal{A}_lx_l(t),\label{abs1}\\ \mathcal{B}_l x_l(t)&=u_l(t),\label{abs2}\\ y_l(t)&= \mathcal{C}_l x_l(t),\label{abs3} 
\end{align}
\end{subequations}
where 
 \begin{align*}
\mathcal{A}_l x_l(t)=\begin{bmatrix} -\gamma {(\rho a)}^{-1} & -EI\partial_{\xi\xi} \\ {(\rho a)}^{-1}\partial_{\xi\xi} & 0 \end{bmatrix}\begin{bmatrix}\rho a \dot{w}_l(\cdot,t) \\ w_l^{\prime\prime}(\cdot,t)\end{bmatrix},
\end{align*} with
\begin{align*}
D(\mathcal{A}_l) = \{x_l \in X_l \ |\  &\mathcal{H}_lx_l \in H^2([-1,0];\mathbb{R}^2),   x_{l2}(-1) = x_{l2}^{\prime}(-1)=0\},
\end{align*}
\begin{align*}
\mathcal{H}_l=\begin{bmatrix} {(\rho a)}^{-1} & 0 \\ 0 & EI\end{bmatrix},\ 
\mathcal{B}_l x_l(t) = \begin{bmatrix} \dot{w}_l(0,t) \\ \dot{w}_l^{\prime}(0,t) \end{bmatrix}\ \text{and} \ 
\mathcal{C}_l x_l(t) = \begin{bmatrix}-EI w_l^{\prime\prime\prime}(0,t) \\ EI w_l^{\prime\prime}(0,t) \end{bmatrix}.
\end{align*}
The operators $\mathcal{B}_l:D(\mathcal{A}_l) \rightarrow U_l$ and $\mathcal{C}_l:D(\mathcal{A}_l)\rightarrow Y_l$ are the virtual control and observation operators with $U_l=\mathbb{R}^2$ and $Y_l=\mathbb{R}^2$.
Here, it is noted that the equations (\ref{abs1}), (\ref{abs2}), (\ref{abs3}) and $D(\mathcal{A}_l)$ corresponds to (\ref{lb1}), (\ref{lb2}), (\ref{lb4}) and (\ref{lb3}), respectively.
The space $X_l$  is a Hilbert space equipped with the energy norm \begin{align*}\|x_l\|_{X_l}^2 := \left\langle x_l,\mathcal{H}_lx_l \right\rangle_{L^2},\ x_l\in X_l.\end{align*} 
Here $\frac{1}{2}\norm{x_l}_{X_l}^2$ is the sum of the kinetic and potential energies of the left beam. The above choice of the state variable corresponds to the port-Hamiltonian formulation of the Euler Bernoulli beam. More details can be found, for example, in \cite{AugJac2014}, \cite{BAugner2018}, and \cite{Augner2018}.

In the same way, the right beam can be written in boundary control and observation form
on the Hilbert space $X_r=L^2([0,1];\mathbb{R}^2)$ with $u_{r1}(t)=\dot{w}_r(0,t),\ u_{r2}(t)=\dot{w}_r^{\prime}(0,t)$ as virtual boundary inputs and $y_{r1}(t)=EI w_r^{\prime\prime\prime}(0,t)$,\ $ y_{r2}(t)=-EI w_r^{\prime\prime}(0,t)$ as virtual outputs. We denote the input and output spaces of the right beam by $U_r=\mathbb{R}^2$ and $Y_r=\mathbb{R}^2$, respectively. Choosing the state variable $x_r(t)=\begin{bmatrix}{\rho a} \dot{w}_r(\cdot,t) \\ w_r^{\prime\prime}(\cdot,t)\end{bmatrix}$, we have
\begin{equation}\label{eq:beam2}
\begin{aligned}
\dot{x}_r(t)&=\mathcal{A}_rx_r(t),\\ \mathcal{B}_r x_r(t)&=u_r(t),\\ y_r(t)&= \mathcal{C}_r x_r(t),
\end{aligned}
\end{equation}
where
\begin{align*}
\mathcal{A}_r &= \begin{bmatrix} -\gamma {(\rho a)}^{-1} & -EI\partial_{\xi\xi} \\ {(\rho a)}^{-1}\partial_{\xi\xi} & 0 \end{bmatrix},
\mathcal{B}_r x_r(t) = \begin{bmatrix} \dot{w}_r(0,t) \\ \dot{w}_r^{\prime}(0,t) \end{bmatrix},
\mathcal{C}_r x_r(t) = \begin{bmatrix}EI w_r^{\prime\prime\prime}(0,t) \\ -EI w_r^{\prime\prime}(0,t) \end{bmatrix} 
\end{align*}
\begin{align*}
\text{and}\quad D(\mathcal{A}_r) = \{x_r \in X_r \ |\  &\mathcal{H}_rx_r \in H^2([0,1];\mathbb{R}^2),  x_{r2}(1)=x_{r2}^{\prime}(1)=0\},
\end{align*}
 $\mathcal{H}_r=\begin{bmatrix} {(\rho a)}^{-1} & 0 \\ 0 & EI\end{bmatrix}$. The space $X_r$ is equipped with the energy norm $\|x_r\|_{X_r}^2 := \left\langle x_r,\mathcal{H}_rx_r \right\rangle_{L^2},\ x_r\in X_r.$

Next, we combine the two beam systems (\ref{eq:beam1}) and (\ref{eq:beam2}) into a single open loop system on the Hilbert space $X_b=X_l \times X_r$ as follows. From the above formulation and from the boundary conditions in \eqref{eq:pdeode}, it is clear that $u_l(t)=u_r(t)$. Now, in order to have the coupling between the beam system and the rigid body as in Figure \ref{interconnection}, the input and the output of the combined beam system are chosen such that the output of the combined beam system is equal to the addition of the outputs of the left and the right beam systems and the input of the combined beam system is equal to the inputs of the left and the right beam systems. Therefore, denoting the input and output spaces of the combined system by $U_b$ and $Y_b$, respectively, let us define a new virtual output function 
\begin{align*}
y_b(t)=\begin{bmatrix}\mathcal{C}_l & \mathcal{C}_r\end{bmatrix}\begin{bmatrix}x_l(t) \\ x_r(t)\end{bmatrix} 
\end{align*}
and a virtual input function 
\begin{align*}
u_b(t)=\begin{bmatrix}\frac{1}{2}\mathcal{B}_l & \frac{1}{2}\mathcal{B}_r\end{bmatrix}\begin{bmatrix}x_l(t) \\ x_r(t)\end{bmatrix}.
\end{align*}
Then the combined system can be written as
\begin{equation}\label{eq:beamsystem}
\dot{x}_b(t)= \mathcal{A}_b x_b(t),\quad \mathcal{B}_bx_b(t)=u_b(t),\quad \mathcal{C}_bx_b(t)=y_b(t)
\end{equation}
where 
\begin{align*}
x_b(t) &=\begin{bmatrix} x_l(t) \\ x_r(t)\end{bmatrix},\ \mathcal{A}_b=\begin{bmatrix} \mathcal{A}_l & 0 \\ 0 & \mathcal{A}_r \end{bmatrix},\ \mathcal{B}_b=\begin{bmatrix}\frac{1}{2}\mathcal{B}_l & \frac{1}{2}\mathcal{B}_r\end{bmatrix},\ \mathcal{C}_b=\begin{bmatrix}\mathcal{C}_l & \mathcal{C}_r\end{bmatrix},\\ 
\text{and}\ D(\mathcal{A}_b)&=\{(x_l,x_r)\in D(\mathcal{A}_l)\times D(\mathcal{A}_r):\mathcal{B}_lx_l=\mathcal{B}_rx_r\}.
\end{align*}

\begin{lemma}
The beam system $(\mathcal{A}_b,\mathcal{B}_b,\mathcal{C}_b)$ in (\ref{eq:beamsystem}) is an impedance passive system on $(X_b,U_b,Y_b)$.
\end{lemma}
\begin{proof}
From \cite[Sec. 2.1]{GovHumPau2020}, we have that the left beam $(\mathcal{A}_l,\mathcal{B}_l,\mathcal{C}_l)$ and the right beam $(\mathcal{A}_r,\mathcal{B}_r,\mathcal{C}_r)$ are impedance passive systems. 
Now, using the boundary condition  
$u_l(t)=u_r(t)$, we obtain
\begin{align*}
\frac{1}{2}\frac{d}{dt}\norm{x_b(t)}_{X_b}^2 &= \frac{1}{2}\frac{d}{dt}\norm{x_l(t)}_{X_l}^2+\frac{1}{2}\frac{d}{dt}\norm{x_r(t)}_{X_r}^2,\\
                                                            &\leq \left\langle u_l(t),y_l(t)\right\rangle_{U_l,Y_l}+\left\langle u_r(t),y_r(t)\right\rangle_{U_r,Y_r}\\
                                                            &= \left\langle u_b(t),y_b(t)\right\rangle_{U_b,Y_b},
\end{align*} 
where $x_b(t),\ t>0$ are solutions of \eqref{eq:beamsystem}. Therefore, (\ref{eq:beamsystem}) is an impedance passive system.
\end{proof}
\begin{remark}\label{rem.Bb}
The impedance passivity of the systems $(\mathcal{A}_l,\mathcal{B}_l,\mathcal{C}_l)$, $(\mathcal{A}_r,\mathcal{B}_r,\mathcal{C}_r)$ and $(\mathcal{A}_b,\mathcal{B}_b,\mathcal{C}_b)$ imply that $A_l=\mathcal{A}_l|_{\mathcal{N}(\mathcal{B}_l)}$, $A_r=\mathcal{A}_r|_{\mathcal{N}(\mathcal{B}_r)}$ and $A_b=\mathcal{A}_b|_{\mathcal{N}(\mathcal{B}_b)}$ generate $C_0$-semigroups of contractions $T_l(t)$, $T_r(t)$ and $T_b(t)$ on $X_l$, $X_r$ and $X_b$, respectively. Therefore, $(\mathcal{A}_l,\mathcal{B}_l,\mathcal{C}_l)$, $(\mathcal{A}_r,\mathcal{B}_r,\mathcal{C}_r)$ and $(\mathcal{A}_b,\mathcal{B}_b,\mathcal{C}_b)$ are boundary control and observation systems \cite[Sec. 4.2]{GorZwaMas2005}. This implies from Remark \ref{unboundedB} that there exist unique operators $B_l\in \mathcal{L}(U_l,X_{l_{-1}})$, $B_r\in \mathcal{L}(U_r,X_{r_{-1}})$ and $B_b\in \mathcal{L}(U_b,X_{b_{-1}})$ such that $\mathcal{A}_l=A_{l_{-1}}+B_l\mathcal{B}_l$ on $D(\mathcal{A}_l)$, $\mathcal{A}_r=A_{r_{-1}}+B_r\mathcal{B}_r$ on $D(\mathcal{A}_r)$ and $\mathcal{A}_b=A_{b_{-1}}+B_b\mathcal{B}_b$ on $D(\mathcal{A}_b)$, respectively. 
\end{remark}
\subsection{The Rigid Body}
Without external inputs, the center rigid body that we extract from the satellite system is given by
\begin{equation}\label{rigidbody}
\begin{aligned}
m\ddot{w}_c(t)&=u_{c1}(t),\\
I_m\ddot{\theta}_c(t)&=u_{c2}(t),\\
y_{c1}(t)&=\dot{w}_c(t),\\ y_{c2}(t)&=\dot{\theta}_c(t),
\end{aligned}        
\end{equation}
 where $u_{c1}(t),\ u_{c2}(t)$ are the virtual inputs and $y_{c1}(t),\ y_{c2}(t)$ are the outputs of the rigid body (see Figure \ref{interconnection}), respectively. The state, input and output spaces of the rigid body are given by $X_c=\mathbb{R}^2,\ U_c=\mathbb{R}^2$ and $Y_c=\mathbb{R}^2$, respectively. Then, with the state variable $x_c(t)= \begin{bmatrix} \dot{w}_c(t) \\ \dot{\theta}_c(t)\end{bmatrix}$, the rigid body (\ref{rigidbody}) on the Hilbert space  $X_c$ can be written as,
\begin{equation}\label{eq:FDS}
\begin{aligned}
\dot{x}_c(t)&=A_c x_c(t)+B_c u_c(t), \\ y_c(t)&=C_c x_c(t),\\
\end{aligned}
\end{equation}
where,
\begin{align*}
A_c=0,\  B_c=\begin{bmatrix}\frac{1}{m} & 0 \\ 0 & \frac{1}{I_m} \end{bmatrix},\   C_c = \begin{bmatrix}1 & 0 \\ 0 & 1 \end{bmatrix},\ \text{and}\ u_c(t)=\begin{bmatrix}u_{c1}(t) \\ u_{c2}(t)\end{bmatrix}.
\end{align*}
The space $X_c$ is equipped with the energy norm 
\begin{align*}
\norm{x_c}_{X_c}^2 = x_c^*\mathcal{H}_c x_c, \quad \text{where}\quad \mathcal{H}_c = \begin{bmatrix} m & 0 \\ 0 & I_m \end{bmatrix}.
\end{align*}
It is straightforward to see that the rigid body is an impedance passive system on $X_c$ (see \cite[Sec. 2.3]{GovHumPau2020}). More details on the energy state space formulation of finite- dimensional systems can be found in \cite[Ch. 2.3]{JacZwa2011}.

\subsection{The Satellite System as a Coupled PDE-ODE System}\label{sec.pde-ode}

From the equations (\ref{eq:beamsystem}) and (\ref{eq:FDS}), we are now ready to write our satellite system (\ref{eq:pdeode})-(\ref{outputs}) as an abstract PDE-ODE system with the power-preserving interconnection $u_b(t)=y_c(t)$, $u_c(t) =-y_b(t)$ (see Figure \ref{PP-interconnection}) on the state space $X = X_b \times X_c$ as 
\begin{equation}\label{eq:satellite}
\begin{aligned}
\begin{bmatrix} \dot{x}_b(t) \\ \dot{x}_c(t) \end{bmatrix}&=\begin{bmatrix} \mathcal{A}_b & 0\\ -B_c \mathcal{C}_b & 0 \end{bmatrix}\begin{bmatrix} x_b(t) \\ x_c(t) \end{bmatrix}+\begin{bmatrix} 0 \\ B_c\end{bmatrix}u(t)+\begin{bmatrix} B_{d0} & 0 \\ 0 & B_c\end{bmatrix}w_d(t),\\
y(t) &= \begin{bmatrix} 0 & C_c\end{bmatrix}\begin{bmatrix} x_b(t) \\ x_c(t) \end{bmatrix},
\end{aligned}
\end{equation}
where $u(t)=\begin{bmatrix} u_1(t)\\u_2(t) \end{bmatrix}$,  $y(t)=\begin{bmatrix} y_1(t)\\y_2(t) \end{bmatrix}$, $w_d(t)=\begin{bmatrix} w_{d1}(t) & w_{d2}(t) & w_{d3}(t) & w_{d4}(t) \end{bmatrix}^T$ and $B_{d0}=\begin{bmatrix} b_{d1}(\cdot) & 0 \\ 0 & 0 \\ 0 & b_{d2}(\cdot) \\ 0 & 0\end{bmatrix}$.
\begin{figure}[!ht]
\centering
\begin{tikzpicture}[auto, node distance=2cm,>=latex']
		\tikzstyle{block} = [draw, fill=cyan!20, rectangle, 
    minimum height=3em, minimum width=6em]
\tikzstyle{sum} = [draw, fill=blue!20, circle, node distance=1cm]
\tikzstyle{input} = [coordinate, node distance=1cm]
\tikzstyle{output} = [coordinate]
\tikzstyle{sinput} = [coordinate]
\tikzstyle{routput} = [coordinate]
\tikzstyle{eoutput} = [coordinate, node distance=1cm]
\tikzstyle{pinstyle} = [pin edge={to-,thin,black}]

    \node [sinput, name=sinput] {};
		\node [input, right of=sinput] (input) {};
    \node [block, right of=input] (Beam System) {Beam System};
		\node [output, right of=Beam System] (output) {};
		\node [eoutput, right of=output] (eoutput) {};
    \node [block, below of=Beam System] (Rigid Body) {Rigid Body};
		\node [routput, below of=input] (routput) {};
		
		\draw  [draw,->] (sinput) -- node[pos=0.9] {$u_b(t)$} (input);
		\draw  (input) -- node {$$} (Beam System);
    \draw  [->] (Beam System) -- node[pos=0.99] {$-y_b(t)$}  
             node [near end] {$$} (output);
		\draw  (output) -- node {} (eoutput);
		\draw [->] (eoutput) |-node[pos=1.8] {$y_c(t)$} 
             node [near end] {$$}  (Rigid Body);
    \draw  (Rigid Body) -|  (sinput) -- (input);
		\draw [->] (Rigid Body) -- (routput);
		\draw (Beam System) -- (eoutput);
		\draw (output) -- (eoutput) |- node[pos=0.8] {$u_c(t)$}   (Rigid Body)  ;
 \end{tikzpicture}
\caption{The interconnection between the beams and the rigid body}
\label{PP-interconnection}
\end{figure}
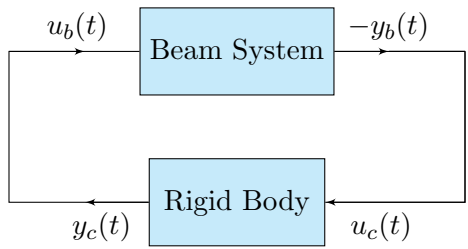

Equation (\ref{eq:satellite}) is in the form (\ref{statespaceform}) with $A=\begin{bmatrix} \mathcal{A}_b & 0\\ -B_c \mathcal{C}_b & 0 \end{bmatrix}$,  $B=\begin{bmatrix} 0 \\ B_c\end{bmatrix}$, $C=\begin{bmatrix} 0 & C_c\end{bmatrix}$, $B_d=\begin{bmatrix} B_{d0} & 0 \\ 0 & B_c\end{bmatrix}$ and $ x(t)=\begin{bmatrix} x_b(t) \\ x_c(t) \end{bmatrix}$. The domain of $A$ is given by
\begin{align*} D(A)=\{(x_b,x_c) \in D(\mathcal{A}_b)\times X_c\  :\  \mathcal{B}_bx_b= C_c x_c\}.\end{align*} 
The norm on $X$ is defined as 
\begin{align*}
\bigg\|{\begin{bmatrix} x_b \\ x_c \end{bmatrix}}\bigg\|_{X}^2=\norm{x_b}_{X_b}^2+\norm{x_c}_{X_c}^2,\quad x_b\in X_b,\ x_c\in X_c.
\end{align*}
\begin{remark}
The operator $A$  is dissipative, since using the power preserving interconnection, we obtain
\begin{align*}
\frac{1}{2}\frac{d}{dt}\bigg\|\begin{bmatrix} x_b(t) \\ x_c(t) \end{bmatrix}\bigg\|_{X}^2 \leq 0.
\end{align*} Therefore, by \cite[Theorem 3.5]{Augner2018},  $A$ generates a $C_0$-semigroup of contractions on $X$.
\end{remark}
\section{Stability of the Satellite Model}\label{sec.sta}
 
In this section, we will show the exponential stability of the satellite system  
in the sense that the operator $A$ defined in Section \ref{sec.pde-ode} generates an exponentially stable semigroup $T(t)$. Let us recall the operator $A$  
\begin{equation}\label{operatorA}
\begin{aligned}
A&=\begin{bmatrix} \mathcal{A}_b & 0\\ -B_c \mathcal{C}_b & 0 \end{bmatrix},\\
D(A)&=\{(x_b,x_c) \in D(\mathcal{A}_b)\times X_c\  :\  \mathcal{B}_bx_b= C_c x_c\}.
\end{aligned}
\end{equation} 
\begin{theorem}\label{mainthm}
The semigroup $T(t)$ generated by $A$ in \eqref{operatorA} is exponentially stable.
\end{theorem}

 We prove the theorem by using frequency domain criteria \cite[Cor. 3.36]{LuoGuoMor99} which states that the semigroup $T(t)$ generated by $A$ is exponentially stable if and only if $i\mathbb{R}\subset \rho(A)$ and $\sup_{\omega\in\mathbb{R}}{\norm{R(i\omega,A)}}<\infty$. We complete the proof in the following steps. Since the satellite system is a coupled system of the beam system and the center rigid body, we will first show that $i\mathbb{R}\subset \rho(A_b)$ and $\sup_{\omega\in\mathbb{R}}{\norm{R(i\omega,A_b)}}<\infty$ where $A_b=\mathcal{A}_b|_{\mathcal{N}(\mathcal{B}_b)}$. As the second step, we will show that $i\mathbb{R}\subset \rho(A)$. In this step, we will obtain an expression for the resolvent operator $R(i\omega,A)$. Next, we will estimate upper bounds of the operators which appear in the resolvent expression. Finally, we will show that $R(i\omega,A)$ is uniformly bounded.

\begin{lemma}\label{lem.observability}
The operator $A_b$ defined in Remark \ref{rem.Bb} satisfies $i\mathbb{R}\subset \rho(A_b)$ and $\sup_{\omega\in\mathbb{R}}{\norm{R(i\omega,A_b)}}<\infty$.
\end{lemma}
\begin{proof}
We show that the semigroup $T_b(t)$ generated by $A_b$ is exponentially stable  
which guarantees $i\mathbb{R}\subset \rho(A_b)$ and uniform boundedness of the resolvent $R(i\omega,A_b)$.
First we claim that the operator $A_r=\mathcal{A}_r|_{\mathcal{N}(\mathcal{B}_r)}$ corresponding to the right beam system (\ref{eq:beam2}) generates an exponentially stable semigroup $T_r(t), t\geq 0$. We use \cite[Main Theorem 1]{Chen1991}.
We write $A_r$ as $A_r=A_0+B_0$ \
where \begin{align*}
A_0=\begin{bmatrix} 0 & -EI\partial_{\xi\xi} \\ {(\rho a)}^{-1}\partial_{\xi\xi} & 0 \end{bmatrix},\quad B_0=\begin{bmatrix} -\gamma (\rho a)^{-1} & 0 \\ 0 & 0 \end{bmatrix}
\end{align*}
and $D(A_0)=D(A_r)$. We will show that the operators $A_0$ and $B_0$ satisfy the following conditions.
\begin{enumerate}
	\item [(c1)] The operator $A_0$ is skew-adjoint and it has compact resolvent.
	\item [(c2)] The spectrum of $A_0$ satisfies the gap property
	\begin{align*}
	\inf{\{|\lambda_j-\lambda_k|\ | j,k=1,2,3,\cdots, j\neq k\}}>0.
	\end{align*}
	\item [(c3)] The operator $B_0$ is dissipative.
	\item [(c4)] If any sequence $\{(x_{r_n})\in X_r,\ n=1,2,\cdots\}$ satisfies \begin{align*}\lim_{n \to \infty}\text{Re}\left\langle B_0 x_{r_n},x_{r_n}\right\rangle_{X_r}=0,\end{align*} then $\lim_{n\to\infty}\norm{B_0 x_{r_n}}_{X_r}=0$.
	\item [(c5)] There exists $\delta>0$ such that $\norm{B_0 \phi_k}_{X_r}\geq \delta$, where $\phi_k,\ k\in \mathbb{Z}$ is an orthonormal eigenvector of $A_0$.
\end{enumerate}
 We have $\text{Re}\left\langle A_0 x_r,x_r \right\rangle=0, x_r \in D(A_0)$. Therefore, by \cite[Thm. 2.3]{AugJac2014}, $A_0$ has compact resolvent. This implies that the operator $A_0$ is skew-adjoint.

By a direct computation, we can obtain eigenvalues $i\lambda_k$ of $A_0$ and orthonormal basis $\phi_k=(f_k,g_k)^T, k\in \mathbb{Z}$ consisting of eigenvectors of $A_0$. The eigenvalues and the eigenvectors are given by
\begin{equation}\label{eigvec}
\begin{aligned}
i\lambda_k &=i \sqrt{\frac{EI}{\rho a}} [\pi(k-\frac{1}{2})+\mathcal{O}(e^{-\pi(k-\frac{1}{2})})]^2, \\
f_k(\xi)&=\beta_k [(\cosh{(\mu_k)}+\cos{(\mu_k)}) (\cosh{(\mu_k\xi)}-\cos{(\mu_k\xi)}) \\ & \qquad \qquad -(\sinh{(\mu_k)}-\sin{(\mu_k)}) (\sinh{(\mu_k\xi)}-\sin{(\mu_k\xi)})],\\
g_k(\xi)&=\frac{\beta_k}{ i\sqrt{\rho a EI}}[(\cosh{(\mu_k)}+\cos{(\mu_k)}) (\cosh{(\mu_k\xi)}+\cos{(\mu_k\xi)})\\ & \qquad \qquad -(\sinh{(\mu_k)}-\sin{(\mu_k)})(\sinh{(\mu_k\xi)}+\sin{(\mu_k\xi)})],
\end{aligned}
\end{equation}
where $\mu_k=(\frac{\rho a}{EI})^{\frac{1}{4}} \sqrt{\lambda_k}$ are the 
 solutions of $\cosh{(\mu_k)}\cos{(\mu_k)}+1=0$ and $\beta_k >0$ are chosen such that $\norm{\phi_k}_{X_r}=1$. 
It is clear that 
the condition (c2) is satisfied since the gap between two successive eigenvalues satisfies $|\lambda_k-\lambda_{k+1}|\rightarrow \infty$ as $k \rightarrow \infty$.
The operator $B_0$ is dissipative since
\begin{align*}
\text{Re}\left\langle B_0x_r,x_r\right\rangle_{X_r} = -\gamma(\rho a)^{-2} \norm{x_{r_1}}_{L^2}^2 \leq 0.
\end{align*}
Also, $-\text{Re}\left\langle B_0x_r,x_r\right\rangle_{X_r}=\gamma^{-1}\rho a \norm{B_0x_r}^2_{X_r}$ holds. This implies that the conditions (c3) and (c4) are satisfied. 

Next, we show that the condition (c5) is satisfied.   
The formulas for $f_k$ and $g_k$ in \eqref{eigvec} can be used to show that 
\begin{equation}\label{eq.lim}
\lim_{|k|\to\infty} \frac{\norm{g_k}_{L^2}}{\norm{f_k}_{L^2}}=\frac{1}{\sqrt{\rho a EI}}.
\end{equation}
 Here we note that $f_k \neq 0,\ \forall k\in \mathbb{Z}$, since $(f_k,g_k)^T$ are  eigenvectors and $f_k=0$ would imply $g_k=\frac{(\rho a)^{-1}}{i\lambda_k}f_k^{\prime\prime}=0$. The equation \eqref{eq.lim} implies that for all $\epsilon>0$, there exists $N \in \mathbb{N}$ such that for all $k\in \mathbb{Z}$ with  $|k|\geq N$, we have
\begin{align*}
\bigg|\frac{\norm{g_k}_{L^2}}{\norm{f_k}_{L^2}}\bigg|\leq \epsilon+\frac{1}{\sqrt{\rho a EI}}.
\end{align*}
Thus
\begin{align*} \frac{\norm{g_k}_{L^2}}{\norm{f_k}_{L^2}} \leq \frac{C'}{\sqrt{\rho a EI}},\  \forall\  k\in \mathbb{Z}, \end{align*}
 where $C'= \max\{1+\epsilon \sqrt{\rho a EI} ,\sqrt{\rho a EI}\max_{|k|<N}\frac{\norm{g_k}_{L^2}}{\norm{f_k}_{L^2}}\}$.   Now we obtain
\begin{align*}
\norm{B_0 \phi_k}_{X_r}^2&= \gamma^2 (\rho a)^{-3}\norm{f_k}_{L^2}^2,\\
                       &\geq\frac{1}{2}\gamma^2 (\rho a)^{-3}(\norm{f_k}_{L^2}^2+\frac{\rho a EI}{C'^2} \norm{g_k}_{L^2}^2 ),\\
											 &\geq\frac{1}{2C'^2}\gamma^2 (\rho a)^{-2}\norm{\phi_k}_{X_r}^2,\\
											 &=\frac{1}{2C'^2}\gamma^2 (\rho a)^{-2} \geq \delta^2 >0, \ \forall k\in \mathbb{Z}.
\end{align*}
Now all the conditions (c1)-(c5) are satisfied. Hence by \cite[Main Theorem 1]{Chen1991}, we have that $A_r$ generates an exponentially stable semigroup $T_r(t)$.

Analogously, we have that $A_l=\mathcal{A}_l|_{\mathcal{N}(\mathcal{B}_l)}$ generates an exponentially stable semigroup $T_l(t),\ t>0$. Thus $A_b$ generates an exponentially stable semigroup $T_b(t)$ which completes the proof.
\end{proof}
\begin{lemma}\label{beamres}
Let $P_b(\cdot)=\mathcal{C}_bR(\cdot,A_{b_{-1}})B_b$ and $P_c(\cdot)=C_cR(\cdot,A_c)B_c$ be the transfer functions of the beam system $(\mathcal{A}_b,\mathcal{B}_b,\mathcal{C}_b)$ and the center rigid body $(A_c,B_c,C_c)$, respectively. Assume that $P_b(0)$ and $I+P_b(i\omega)P_c(i\omega)$, $\omega \in \mathbb{R}\backslash\{0\}$ are nonsingular. Then the operator $A$ in \eqref{operatorA} satisfies $i\mathbb{R}\subset \rho(A)$.
\end{lemma}
\begin{proof}
We will show that the operator $i\omega-A$  
is bijective. 
Let $\omega \in \mathbb{R}$ be arbitrary. We start by proving $i\omega-A$ is injective. Let $(x_b,x_c)^T \in D(A)=\{(x_b,x_c) \in D(\mathcal{A}_b)\times X_c\  :\  \mathcal{B}_bx_b= C_c x_c\}$ be such that $(i\omega-A)(x_b,x_c)^T=0$. Then by using the structure of $A$, we obtain 
\begin{align*}
\begin{bmatrix} (i\omega-\mathcal{A}_b)x_b \\ B_c\mathcal{C}_bx_b+i\omega x_c\end{bmatrix}=0.
\end{align*}
We have from Lemma \ref{lem.observability} that $i\mathbb{R}\subset \rho(A_b)$. By using Remark \ref{rem.Bb}, solving the above equation, we obtain 
\begin{equation}\label{injective}
\begin{aligned}
x_b=R(i\omega,A_{b_{-1}})B_bC_c x_c,\\
[i\omega I_{X_c}+B_c P_b(i\omega)C_c]x_c=0.
\end{aligned}
\end{equation}
We have that $B_c$, $C_c$ are nonsingular and $P_b(0)$ and $I+P_b(i\omega)P_c(i\omega)$ are assumed to be nonsingular. 
Therefore, the function
\begin{equation}\label{schcom}
\begin{aligned}
S(i\omega) =\begin{cases} \frac{1}{i\omega}+\frac{1}{\omega^2} B_c P_b(i\omega) (I+P_b(i\omega)P_c(i\omega))^{-1} C_c,\ \omega \in \mathbb{R}\backslash\{0\},\\
	(B_c P_b(0)C_c)^{-1},\ \omega=0\end{cases}
\end{aligned}
\end{equation}
is well-defined for all $\omega\in\mathbb{R}$.
A direct computation shows that $S(i\omega) = [i\omega I_{X_c}+B_c P_b(i\omega)C_c]^{-1}$ for all $\omega\in \mathbb{R}$. This implies by \eqref{injective} that $(x_b,x_c)=0$. Thus, the operator $i\omega-A$ is injective.

Now it remains to prove that $i\omega-A$ is surjective. For all $f_b\in X_b$ and $f_c \in X_c$, our aim is to find $(x_b,x_c)^T \in D(A)$ such that
\begin{equation} \label{reseq1}
\begin{aligned}
\begin{bmatrix} f_b \\ f_c \end{bmatrix}&=(i\omega-A)\begin{bmatrix} x_b  \\ x_c\end{bmatrix}=\begin{bmatrix} (i\omega-\mathcal{A}_b)x_b \\ B_c\mathcal{C}_bx_b+i\omega x_c\end{bmatrix}.
\end{aligned}
\end{equation} 
Since $i\mathbb{R}\subset \rho(A_b)$, using Remark \ref{rem.Bb}, the solution of (\ref{reseq1}) is given by
\begin{equation}\label{resderivation}
\begin{aligned}
x_b  &=[R(i\omega,A_b)-R(i\omega,A_{b_{-1}})B_b C_cS(i\omega)B_c C_b R(i\omega,A_b)]f_b \\ &\quad +R(i\omega,A_{b_{-1}})B_b C_cS(i\omega) f_c\\
	x_c &=-S(i\omega)B_cC_bR(i\omega,A_b)f_b +S(i\omega)f_c
\end{aligned}
\end{equation}
where $C_b=\mathcal{C}_b|_{\mathcal{N}(\mathcal{B}_b)}$. Moreover, for $(x_b,x_c) \in D(\mathcal{A}_b) \times X_c$, we have
\begin{align*}
&\begin{bmatrix}\mathcal{B}_b & -C_c\end{bmatrix}\begin{bmatrix}x_b \\ x_c\end{bmatrix}\\ & =\mathcal{B}_bR(i\omega,A_b)f_b-\mathcal{B}_bR(i\omega,A_{b_{-1}})B_bC_c S(i\omega)B_c C_b R(i\omega,A_b)f_b \\ &+\mathcal{B}_bR(i\omega,A_{b_{-1}})B_bC_c S(i\omega)f_c +C_cS(i\omega)B_c C_bR(i\omega,A_b)f_b-C_cS(i\omega)f_c \\
&=0
\end{align*}
since $\mathcal{B}_bR(i\omega,A_{b_{-1}})B_b=I$ \cite[Prop. 10.1.2]{JacZwa2011} and $\mathcal{R}(R(i\omega,A_b))\subset D(A_b)$. Thus $(x_b,x_c)^T \in D(A)$. This implies that the operator $i\omega-A$ is surjective. Thus $i\omega-A,\ \omega\in \mathbb{R}$ has a bounded inverse, which completes the proof.
\end{proof} 
\begin{remark}
From the equation \eqref{resderivation} in Lemma \ref{beamres}, the resolvent operator $R(i\omega,A)$ has the form
\begin{equation}\label{reseqn}
R(i\omega,A)=\begin{bmatrix} R_{11}(i\omega) &  R_{12}(i\omega)\\ R_{21}(i\omega)  & R_{22}(i\omega) \end{bmatrix}
\end{equation}
where
\begin{equation}\label{subreseqn}
\begin{aligned}
 R_{11}(i\omega) &=R(i\omega,A_b)-R(i\omega,A_{b_{-1}})B_bC_c S(i\omega)B_c C_b R(i\omega,A_b),\\
R_{12}(i\omega)&=R(i\omega,A_{b_{-1}})B_bC_c S(i\omega),\\
R_{21}(i\omega)&=-S(i\omega)B_c C_bR(i\omega,A_b),\\
R_{22}(i\omega)&=S(i\omega).
\end{aligned}
\end{equation}
\end{remark}
In the following we derive an upper bound for the transfer function $P_b(i\omega)$ and upper bounds for the operators $R(i\omega,A_{b_{-1}})B_b$, $C_b R(i\omega,A_b)$ and $(I+P_b(i\omega)P_c(i\omega))^{-1}$. 
\begin{lemma}\label{lemtf}
Let $P_b(\cdot)$ be the transfer function of the beam system $(\mathcal{A}_b, \mathcal{B}_b, \mathcal{C}_b)$. Then there exists $M>0$ such that $\|P_b(i\omega)\| \leq M(|\omega|+1)$ for all $\omega \in \mathbb{R}$. Moreover, $P_b(0)$ is nonsingular.
\end{lemma}
\begin{proof}
For $u_b\in U_b$, the transfer function of the beam system $(\mathcal{A}_b, \mathcal{B}_b, \mathcal{C}_b)$ is given by
\begin{align*}
P_b(i\omega)u_b=P_l(i\omega)u_l+P_r(i\omega)u_r,\ \omega\in \mathbb{R},
\end{align*}
where $P_l(i\omega)$ and $P_r(i\omega)$ are the transfer functions of the left and the right beam systems, respectively, and we will now derive an explicit expression for them. For $u_r\in U_r$, the transfer function $P_r(i\omega)$ of the right beam system $(\mathcal{A}_r, \mathcal{B}_r, \mathcal{C}_r)$ can be obtained as the unique solution of 
\begin{align*}
(i\omega-\mathcal{A}_r) x_r = 0,\\
 \mathcal{B}_rx_r=u_r \\
P_r(i\omega)u_r=\mathcal{C}_rx_r
\end{align*}
with $x_r \in D(\mathcal{A}_r)=\{x_r=(f_r,g_r)^T \in X_r \ |\  \mathcal{H}_rx_r \in H^2([0,1];\mathbb{R}^2),\  g_r(1)=g_r^{\prime}(1)=0\}$ (\cite[Thm. 12.1.3]{JacZwa2011}). Replacing the operators with the corresponding expressions, the above equations can be written as 
\begin{equation}\label{eqbvp}
\begin{aligned}
 (i\omega+\gamma (\rho a)^{-1})f_r+EI g_r^{\prime\prime}=0, \\
-(\rho a)^{-1}f_r^{\prime\prime}+i\omega g_r =0 ,\\
 f_r(0)= \rho a\  u_{r1}, \ f_r'(0)=\rho a\ u_{r2},\\
EI g_r(1) = 0,\ EI g_r'(1)=0,\\
P_r(i\omega)u_r=EI\begin{bmatrix} g_r'(0) \\ -g_r(0) \end{bmatrix}.
\end{aligned}
\end{equation}
We consider the case $\omega=0$ separately. Solving (\ref{eqbvp}) for $\omega=0$, we obtain
\begin{align*}
f_r(\xi)&=\rho a (u_{r1}+\xi u_{r2}),\\
g_r(\xi)&=\frac{\gamma}{EI}\bigg[\bigg(\frac{-\xi^2}{2}+\xi-\frac{1}{2}\bigg)u_{r1}+\bigg(\frac{-\xi^3}{6}+\frac{\xi}{2}-\frac{1}{3}\bigg)u_{r2}\bigg]
\end{align*}
and therefore $P_r(0)$ is given by
\begin{align*}
P_r(0)u_r=EI\begin{bmatrix} g_r'(0) \\ -g_r(0) \end{bmatrix}=\gamma \begin{bmatrix} 1 & \frac{1}{2} \\ \frac{1}{2} & \frac{1}{3}  \end{bmatrix}\begin{bmatrix} u_{r1} \\ u_{r2}\end{bmatrix}.
\end{align*}
Similarly, we obtain that $P_l(0)$ is given by
\begin{align*}
P_l(0)u_l= EI\begin{bmatrix} -g_l'(0) \\ g_l(0) \end{bmatrix}=\gamma \begin{bmatrix} 1 & \frac{-1}{2} \\ \frac{-1}{2} & \frac{1}{3}  \end{bmatrix}\begin{bmatrix} u_{l1} \\ u_{l2}\end{bmatrix}.
\end{align*}
Now using the boundary conditions $u_{l1}=u_{r1}, u_{l2}=u_{r2}$, the transfer function of the combined beam system is given by,
\begin{equation}\label{beamtf}
P_b(0)u_b=EI\begin{bmatrix} g_r'(0)-g_l'(0) \\ -g_r(0)+g_l(0) \end{bmatrix}=\gamma \begin{bmatrix} 2 & 0 \\ 0 & \frac{2}{3}  \end{bmatrix}u_b.
\end{equation}
Thus $P_b(0)$ is indeed nonsingular.
For $\omega \in \mathbb{R}\backslash\{0\}$, solving (\ref{eqbvp}), we obtain
 \begin{align*}
f_r(\xi)&=\rho a\bigg[(C_{1,\omega}u_{r1} -\frac{C_{3,\omega}}{\alpha(\omega)}u_{r2})f_1(\xi)+(C_{2,\omega}u_{r1}+\frac{C_{4,\omega}}{\alpha(\omega)}u_{r2})f_2(\xi) \\ & \qquad \qquad + \cos{(\alpha(\omega) \xi)}u_{r1}+ \frac{\sin{(\alpha(\omega) \xi)}}{\alpha(\omega)} u_{r2}\bigg],\\
g_r(\xi)&=\frac{{\alpha(\omega)^2}}{i\omega} \bigg[( C_{1,\omega}u_{r1} -\frac{C_{3,\omega}}{\alpha(\omega)}u_{r2})g_1(\xi)+( C_{2,\omega}u_{r1}+\frac{C_{4,\omega}}{\alpha(\omega)}u_{r2})g_2(\xi) \\ & \qquad \qquad - \cos{(\alpha(\omega) \xi)}u_{r1}-\frac{\sin{(\alpha(\omega) \xi)}}{\alpha(\omega)} u_{r2}\bigg],
\end{align*}
where
\begin{equation}\label{eq.C2C3}
\begin{aligned}
C_{1,\omega}&=\frac{C_1(\omega)}{C_2(\omega)},\
C_{2,\omega}=\frac{C_2(\omega)\cos{(\alpha(\omega))}-C_1(\omega)C_5(\omega)}{C_2(\omega)C_3(\omega)},\\
C_{3,\omega}&=\frac{C_4(\omega)}{C_2(\omega)},\
C_{4,\omega}=\frac{C_2(\omega) \sin{(\alpha(\omega))}+C_4(\omega)C_5(\omega)}{C_2(\omega)C_3(\omega)},
\end{aligned}
\end{equation}
\begin{align*}
f_1(\xi)&=\cosh{(\alpha(\omega) \xi)}-\cos{(\alpha(\omega) \xi}),\ f_2(\xi)=\sinh{(\alpha(\omega) \xi)}-\sin{(\alpha(\omega) \xi}),\\
g_1(\xi)&=\cosh{(\alpha(\omega) \xi)}+\cos{(\alpha(\omega) \xi}),\ g_2(\xi)=\sinh{(\alpha(\omega) \xi)}+\sin{(\alpha(\omega) \xi})
\end{align*} and
\begin{align*}
C_1(\omega)&=1 +\cos{(\alpha(\omega))}\cosh{(\alpha(\omega))}+\sin{(\alpha(\omega))}\sinh{(\alpha(\omega))},\\
C_2(\omega)&=2+2 \cosh{(\alpha(\omega))}\cos{(\alpha(\omega))},\\
C_3(\omega)&=\sinh{(\alpha(\omega))}+\sin{(\alpha(\omega))},\\
C_4(\omega)&=\cos{(\alpha(\omega))}\sinh{(\alpha(\omega))}-\sin{(\alpha(\omega))}\cosh{(\alpha(\omega))},\\
C_5(\omega)&=\cosh{(\alpha(\omega))}+\cos{(\alpha(\omega))},\\
\alpha(\omega)&= \bigg(\frac{\rho a}{EI} \bigg)^{\frac{1}{4}}(\omega^2-i\gamma (\rho a)^{-1}\omega)^{\frac{1}{4}}.
\end{align*}
Therefore, the transfer function of the right beam can be written as,
\begin{align*}
P_r(i\omega)u_r=EI\begin{bmatrix} g_r'(0) \\ -g_r(0) \end{bmatrix}, \quad \omega \in \mathbb{R} \backslash \{0\}
\end{align*}
where
\begin{align*}
g_r'(0)&= 2\frac{{\alpha(\omega)}^3}{i\omega} C_{2,\omega}u_{r1} + \frac{{\alpha(\omega)}^2}{i\omega} (2C_{4,\omega}-1) u_{r2},\\
g_r(0)&= \frac{{\alpha(\omega)}^2}{i\omega} (2C_{1,\omega}-1)u_{r1} -2 \frac{{\alpha(\omega)}}{i\omega} C_{3,\omega} u_{r2}.
\end{align*} \normalsize
In the same way, we can obtain the transfer function of the left beam which is given by,
\begin{align*}
P_l(i\omega)u_l=EI\begin{bmatrix} -g_l'(0) \\ g_l(0) \end{bmatrix},\quad \omega \in \mathbb{R} \backslash \{0\}
\end{align*}
where
\begin{align*}
g_l'(0)&= -2\frac{{\alpha(\omega)}^3}{i\omega} C_{2,\omega}u_{l1} + \frac{{\alpha(\omega)}^2}{i\omega} (2C_{4,\omega}-1) u_{l2},\\
g_l(0)&= \frac{{\alpha(\omega)}^2}{i\omega} (2C_{1,\omega}-1)u_{l1} + 2\frac{{\alpha(\omega)}}{i\omega}C_{3,\omega} u_{l2}.
\end{align*} \normalsize
Thus, the transfer function of the combined beam system is given by,
\begin{align}\label{beamtf1}
P_b(i\omega)u_b=4EI\frac{{\alpha(\omega)}}{i\omega}\begin{bmatrix} \alpha(\omega)^2 C_{2,\omega} & 0 \\ 0 &  C_{3,\omega}\end{bmatrix} u_b,\quad \omega \in \mathbb{R} \backslash \{0\}.
\end{align}  
Now, let us estimate the absolute values of $C_{2,\omega}$ and $C_{3,\omega}$ which contain trigonometric and hyperbolic terms.
Writing $\alpha(\omega)$ in terms of its real and imaginary parts, we obtain
\begin{equation}\label{eq.alpha}
\begin{aligned}
\alpha(\omega) &= \bigg(\frac{\rho a}{EI} \bigg)^{\frac{1}{4}}(\omega^2-i\gamma (\rho a)^{-1}\omega)^{\frac{1}{4}},\\
       &= |\alpha(\omega)|\bigg(\cos{\bigg(\frac{\theta(\omega)+2\pi k}{4}\bigg)}+i \sin{\bigg(\frac{\theta(\omega)+2\pi k}{4}\bigg)}\bigg),\ k=0,1,2,3,
\end{aligned}
\end{equation}
where
\begin{align*}
|\alpha(\omega)|&=\bigg(\frac{\rho a}{EI}|\omega|\sqrt{\omega^2+\gamma^2(\rho a)^{-2}}\bigg)^{\frac{1}{4}},\\
\theta(\omega) &= \tan^{-1}\bigg(\frac{-\gamma (\rho a)^{-1}}{\omega}\bigg).
\end{align*}
We have 
\begin{align*}
\text{Re}(\alpha(\omega))&= \pm |\alpha(\omega)| \cos{\bigg(\frac{\theta(\omega)}{4}\bigg)} \ \text{or}\ \text{Re}(\alpha(\omega))=\mp |\alpha(\omega)| \sin{\bigg(\frac{\theta(\omega)}{4}\bigg)},\\
\text{Im}(\alpha(\omega))&= \pm |\alpha(\omega)| \sin{\bigg(\frac{\theta(\omega)}{4}\bigg)} \ \text{or}\ \text{Im}(\alpha(\omega))=\pm |\alpha(\omega)| \cos{\bigg(\frac{\theta(\omega)}{4}\bigg)}.
\end{align*}
In addition, there exists $\omega_1 \geq \gamma (\rho a)^{-1}>0$ such that $\tan^{-1}(\frac{\gamma (\rho a)^{-1}}{|\omega|})\leq \frac{\gamma (\rho a)^{-1}}{|\omega|}$ for all $|\omega|>\omega_1$. 
Therefore, there exist $M_1, M_2, M_3, M_4>0$ and $\omega_2>\omega_1$ such that 
\begin{align*}
M_1 \sqrt{|\omega|}\leq |\alpha(\omega)|\bigg|\cos{\bigg(\frac{\theta(\omega)}{4}\bigg)}\bigg|&\leq M_2 \sqrt{|\omega|}\\
M_3 \frac{1}{\sqrt{|\omega|}}\leq|\alpha(\omega)|\bigg|\sin{\bigg(\frac{\theta(\omega)}{4}\bigg)}\bigg|&\leq M_4 \frac{1}{\sqrt{|\omega|}} 
\end{align*}
for all $|\omega|\geq \omega_2$. Denoting $x_{\omega}=\re(\alpha(\omega))$ and $y_{\omega}=\im(\alpha(\omega))$, the above estimates imply that when $|x_{\omega}|$ grows at a rate of $\sqrt{|\omega|}$, $|y_{\omega}|$ decays at a rate of $\frac{1}{\sqrt{|\omega|}}$ or the other way around. 
Since $\abs{C_{2,\omega}}$ and $\abs{C_{3,\omega}}$ have similar terms for all the four roots of $\alpha(\omega)$, we restrict our analysis to the principal branch of the fourth root of $\alpha(\omega)$ and note that the other branches can be treated similarly. 

 The definition of $\alpha(\omega)$ and straightforward estimates can be used to verify that $\abs{\cosh(\alpha(\omega))\cos(\alpha(\omega))}\to \infty$ as $\abs{\omega}\to \infty$. Therefore, there exists $\omega_0>\omega_2$ such that
\begin{equation}\label{eq.coscosh}
|\cosh{\alpha(\omega)}\cos{\alpha(\omega)}|\geq 2 
\end{equation} for all $|\omega|\geq \omega_0$ and this further implies that 
\begin{equation}\label{eq.bd}
\begin{split}
\bigg|\frac{\sin{(\alpha(\omega))}\sinh{(\alpha(\omega))}}{1+\cos{(\alpha(\omega))}\cosh{(\alpha(\omega))}}\bigg|
&\leq\frac{\abs{\sin{(\alpha(\omega))}\sinh{(\alpha(\omega))}}}{\abs{\cos{(\alpha(\omega))}\cosh{(\alpha(\omega))}}-1}\\
&\leq 2\abs{\tan{(\alpha(\omega))}}\abs{\tanh{(\alpha(\omega))}}\\
&\leq 2(\abs{\coth{(y_{\omega})}}+\abs{\tanh{(y_{\omega})}})\\ & \qquad (\abs{\tanh{(x_{\omega})}}+\abs{\coth{(x_{\omega})}} )
\end{split}
\end{equation}
where  
the last inequality is obtained by separating real and imaginary parts of the second inequality and using straightforward estimates. Here we note that $\abs{\tanh{(x_{\omega})}}$, $\abs{\coth{(x_{\omega})}}$ and $\abs{\tanh{(y_{\omega})}}$ are all uniformly bounded for $|\omega|\geq \omega_0$ and since $|y_{\omega}|$ decays at a rate of $\frac{1}{\sqrt{|\omega|}}$, using Taylor series, we can estimate, there exists $M_0'>0$ such that
\begin{align*}
|\coth{(y_{\omega})}|&=\bigg|y_{\omega}^{-1}+\frac{y_{\omega}}{3}-\frac{y_{\omega}^3}{45}+\cdots\bigg|
                          \leq M_0' \sqrt{|\omega|} 
\end{align*}
for $|\omega|\geq \omega_0$.  
Therefore, from \eqref{eq.bd}, we obtain 
\begin{equation} \label{ub1}
\bigg|\frac{\sinh{(\alpha(\omega))}\sin{(\alpha(\omega))}}{1+\cosh{(\alpha(\omega))}\cos{(\alpha(\omega))}}\bigg| 
\lesssim \sqrt{|\omega|}
\end{equation} for $|\omega|\geq \omega_0$. 
Moreover, $\abs{\frac{\sin{(x_{\omega})}\cosh{(y_{\omega})}}{\sinh{(x_{\omega})}\cos{(y_{\omega})}}}\to 0$ as $|\omega| \to \infty$. This implies that there exists $M_0^{\prime\prime}>0$ such that
\begin{align*}
\abs{\sinh{(\alpha(\omega))}+\sin{(\alpha(\omega))}}&\geq \abs{\re(\sinh{(\alpha(\omega))}+\sin{(\alpha(\omega))})} \\
                                &= \abs{\sinh{(x_{\omega})}\cos{(y_{\omega})}+\sin{(x_{\omega})}\cosh{(y_{\omega})}}\\
																&=\abs{\sinh{(x_{\omega})}\cos{(y_{\omega})}}\bigg|1+\frac{\sin{(x_{\omega})}\cosh{(y_{\omega})}}{\sinh{(x_{\omega})}\cos{(y_{\omega})}}\bigg|\\
																&\geq M_0^{\prime\prime} \abs{\sinh{(x_{\omega})}\cos{(y_{\omega})}} 
\end{align*} for $|\omega|\geq \omega_0$.  
Since $|\cos{(\alpha(\omega))}|$ and $|\tan{(y_{\omega})}|$ are uniformly bounded for $|\omega|\geq \omega_0$, the above estimate implies that there exist $M_1'>0$ and $M_1^{\prime\prime}>0$ such that
\begin{equation}\label{ub2}
\bigg|\frac{\cos{(\alpha(\omega))}}{\sinh{(\alpha(\omega))}+\sin{(\alpha(\omega))}}\bigg|
																			\leq M_1'\qquad \text{and} 
\end{equation}
  
\begin{equation}\label{ub3}
\begin{split}
\bigg|\frac{\cosh{(\alpha(\omega))}}{\sinh{(\alpha(\omega))}+\sin{(\alpha(\omega))}}\bigg|
                      &\leq \frac{1}{M_0^{\prime\prime}} \bigg|\frac{\cosh{(x_{\omega})}\cos{(y_{\omega})}+i\sinh{(x_{\omega})}\sin{(y_{\omega})}}{\sinh{(x_{\omega})}\cos{(y_{\omega})}}\bigg|\\ 
										 &\leq \frac{1}{M_0^{\prime\prime}} [|\coth{(x_{\omega})}|+|\tan{(y_{\omega})}|] 
										 \leq M_1^{\prime\prime}
\end{split}
\end{equation} for $|\omega|\geq \omega_0$. We note that $\abs{\coth{(x_{\omega})}}$ is uniformly bounded for $|\omega|\geq \omega_0$.
Using the estimates \eqref{ub1}, \eqref{ub2} and \eqref{ub3}, from \eqref{eq.C2C3}, we obtain 
\begin{equation}\label{eq.C2}
\begin{split}
\abs{C_{2,\omega}} &\leq \bigg|\frac{\cos{(\alpha(\omega))}}{\sinh{(\alpha(\omega))}+\sin{(\alpha(\omega))}}\bigg|\\ 
																	                       & \qquad+\bigg|\frac{1}{2}+\frac{\sin{(\alpha(\omega))}\sinh{(\alpha(\omega))}}{2+2 \cosh{(\alpha(\omega))}\cos{(\alpha(\omega))}}\bigg|  \bigg|\frac{\cosh{(\alpha(\omega))}+\cos{(\alpha(\omega))}}{\sinh{(\alpha(\omega))}+\sin{(\alpha(\omega))}}\bigg|\\
																												&\lesssim \sqrt{|\omega|}
\end{split}
\end{equation} 
for $|\omega|\geq \omega_0$. 
\normalsize
 Again using \eqref{eq.coscosh}, from \eqref{eq.C2C3}, we can estimate
\begin{align*}
\abs{C_{3,\omega}}&=\bigg|\frac{\cos{(\alpha(\omega))}\sinh{(\alpha(\omega))}-\sin{(\alpha(\omega))}\cosh{(\alpha(\omega))}}{2+2 \cosh{(\alpha(\omega))}\cos{(\alpha(\omega))}}\bigg|\\
                  &\leq \bigg|\frac{\cos{(\alpha(\omega))}\sinh{(\alpha(\omega))}-\sin{(\alpha(\omega))}\cosh{(\alpha(\omega))}}{\cosh{(\alpha(\omega))}\cos{(\alpha(\omega))}}\bigg|\\
										 &\leq |\tanh{(\alpha(\omega))}|+|\tan{(\alpha(\omega))}|\\
										 &\leq \abs{\tanh{(x_{\omega})}}+\abs{\coth{(x_{\omega})}}+\abs{\coth{(y_{\omega})}}+\abs{\tanh{(y_{\omega})}}.
\end{align*} Since $\abs{\tanh{(x_{\omega})}}$, $\abs{\coth{(x_{\omega})}}$ and $\abs{\tanh{(y_{\omega})}}$ are uniformly bounded and $\abs{\coth{(y_{\omega})}}\leq M_0'\sqrt{|\omega|}$  for $|\omega|\geq\omega_0$, we have  
\begin{equation}\label{eq.C3}\abs{C_{3,\omega}}\lesssim \sqrt{|\omega|}\end{equation} for $|\omega|\geq \omega_0$.
Finally, from the estimates \eqref{eq.C2}, \eqref{eq.C3} and from equation (\ref{beamtf1}) we obtain
\begin{align*}
\|P_b(i\omega)u_b\|^2 &\lesssim 16(EI)^2\bigg[\frac{{|\alpha(\omega)|}^6}{|\omega|^2} |\omega| |u_{b1}|^2+ \frac{{|\alpha(\omega)|^2}}{|\omega|^2} |\omega||u_{b2}|^2\bigg],\\
                    &\lesssim (|\omega|+1)^2 |u_b|^2 
\end{align*} for $|\omega|\geq \omega_0$. 
Hence $\|P_b(i\omega)\| \lesssim |\omega|+1$ for all $|\omega|\geq \omega_0$.  Finally, by the continuity of the transfer function $P_b(\cdot)$ on $i\mathbb{R}$, we conclude that $\|P_b(i\omega)\| \lesssim |\omega|+1$ for all $\omega \in \mathbb{R}$.
\end{proof}
\begin{lemma}\label{corresolbd1}
There exists $C'>0$ such that $\norm{R(i\omega,A_{b_{-1}})B_b}\leq C'\sqrt{|\omega|+1}$ for all $\omega \in \mathbb{R}$. Moreover, $I+P_b(i\omega)P_c(i\omega)$ is nonsingular for all $\omega \in \mathbb{R}\backslash \{0\}$.
\end{lemma}
\begin{proof}
By using \cite[Rem. 10.1.5]{TucWei2009}, we have that for every $u_b \in U_b$, $i\omega \in \rho(A_{b_{-1}})$,
\begin{align*}
x_b = R(i\omega,A_{b_{-1}})B_bu_b = \begin{bmatrix}R(i\omega,A_{l_{-1}})B_l \\ R(i\omega,A_{r_{-1}})B_r\end{bmatrix}u_b \in D(\mathcal{A}_b)
\end{align*}
where $B_l$, $B_r$ and $B_b$ are defined in Remark \ref{rem.Bb}, is the unique solution of the abstract elliptic problem 
\begin{align*}(i\omega-\mathcal{A}_b)x_b=0,\\
\mathcal{B}_bx_b=u_b.\end{align*} 
Assume that $|\omega|\geq 1$. Let us start by estimating the norm of $x_r=R(i\omega,A_{r_{-1}})B_ru_b$ which is the unique solution of $(i\omega-\mathcal{A}_r)x_r=0,\ 
\mathcal{B}_rx_r=u_b.$    
If $x_r=(f_r,g_r)^T$, then using the expression for $\mathcal{A}_r$, we have 
\begin{align}
(i\omega+\gamma (\rho a)^{-1})f_r+EI g_r^{\prime\prime}=0, \label{pde1}\\
-(\rho a)^{-1}f_r^{\prime\prime}+i\omega g_r =0 \label{pde2},\\
f_r(0)= \rho a\  u_{b1}, \ f_r'(0)=\rho a\ u_{b2},\\
EI g_r(1) = 0,\ EI g_r'(1)=0.
\end{align}
 Taking $L^2$ inner product of (\ref{pde1}) with $(\rho a)^{-1} f_r $ and $L^2$ inner product of (\ref{pde2}) with $ EI g_r $, respectively, we obtain
\begin{equation}
\begin{aligned}\label{pde1sol}
(\rho a)^{-1}(i\omega+\gamma (\rho a)^{-1}) \norm{f_r}^2_{L^2}-EI(\rho a)^{-1}(\rho a) g_r'(0) u_{b1}\\ -EI(\rho a)^{-1}\int_0^1{g_r'\bar{f_r}'} d\xi =0 .
\end{aligned}
\end{equation}
\begin{align}\label{pde2sol}
EI\bar{g_r}(0)u_{b2}+EI (\rho a)^{-1} \int_0^1{f_r'\bar{g_r}'} d\xi+i EI \omega \norm{g_r}^2_{L^2}=0.
\end{align}
Adding complex conjugate of (\ref{pde2sol}) to (\ref{pde1sol}), we obtain
\begin{equation}\label{rePb}
(\rho a)^{-1}(i\omega+\gamma (\rho a)^{-1}) \norm{f_r}^2_{L^2}-i EI\omega \norm{g_r}^2_{L^2}=\left\langle y_r,u_b\right\rangle.
\end{equation}
Equating real and imaginary parts and using the Cauchy-Schwartz inequality, we obtain
\begin{align*}
\gamma (\rho a)^{-2} \norm{f_r}^2_{L^2} &\leq \norm{P_r(i\omega)u_b}\norm{u_b}\\
EI \norm{g_r}^2_{L^2} &\leq \bigg(\frac{\rho a}{\gamma}+\frac{1}{|\omega|}\bigg)\norm{P_r(i\omega)u_b}\norm{u_b},
\end{align*}
where $P_r(\cdot)$ is the transfer function of the right beam system. Therefore,
\begin{align*}
\norm{x_r}_{X_r}^2 &=(\rho a)^{-1}\norm{f_r}_{L^2}^2+EI \norm{g_r}_{L^2}^2,\\
             &\leq \frac{\rho a}{\gamma} \norm{P_r(i\omega)u_b}\norm{u_b}+\bigg(\frac{\rho a}{\gamma}+\frac{1}{|\omega|}\bigg)\norm{P_r(i\omega)u_b}\norm{u_b},\\
						 &\leq \bigg(\frac{2 \rho a}{\gamma}+1\bigg)\norm{P_r(i\omega)u_b}\norm{u_b}.
\end{align*}
Since we have from Lemma \ref{lemtf} that $\norm{P_r(i\omega)}$ can grow at most linearly, the above estimate implies that there exists $C_1>0$ such that $\norm{x_r}=\norm{R(i\omega,A_{r_{-1}})B_ru_b}\leq C_1 \sqrt{|\omega|+1}\norm{u_b},\ |\omega|\geq 1$. We can analogously show that there exists $C_2>0$ such that $\norm{R(i\omega,A_{l_{-1}})B_lu_b}\leq C_2 \sqrt{|\omega|+1}\norm{u_b},\ |\omega|\geq 1$. Combining these estimates, we can see that $\norm{R(i\omega,A_{b_{-1}})B_b}\lesssim \sqrt{|\omega|+1}$ for all $|\omega|\geq 1$. Finally, by continuity of $R(i\omega,A_{b_{-1}})B_b$ with respect to $i\omega$ on $i\mathbb{R}$, we have that 
$\norm{R(i\omega,A_{b_{-1}})B_b}\lesssim \sqrt{|\omega|+1}$ for all $\omega\in\mathbb{R}$.

From equation (\ref{rePb}), we observe that $\re{P_r(i\omega)}>0,\ \omega\in\mathbb{R}$. Indeed, from (\ref{rePb}) we have 
\begin{align*}
\re \left\langle y_r,u_b\right\rangle = \re\left\langle P_r(i\omega)u_b,u_b\right\rangle=\gamma (\rho a)^{-2} \norm{f_r}^2_{L^2}.
\end{align*}
Analogously, we have that $\re{P_l(i\omega)}>0,\ \omega\in\mathbb{R}.$ This implies that $\re{P_b(i\omega)}>0,\ \omega\in\mathbb{R}$. In addition, from the transfer function
\begin{equation}\label{tfrigidbody}P_c(i\omega)=\frac{1}{i\omega}\begin{bmatrix} \frac{1}{m} & 0 \\ 0 & \frac{1}{I_m}\end{bmatrix},\ \omega \in \mathbb{R}\backslash \{0\}\end{equation} of the rigid body, we see that $\re{P_c(i\omega)}=0,\ \omega\in\mathbb{R}\backslash\{0\}$. 
Consequently, we have that $I+P_b(i\omega)P_c(i\omega)$ is nonsingular for all $\omega \in \mathbb{R}\backslash \{0\}$.
\end{proof}

\begin{lemma}\label{corresolbd2}
There exists $C^{\prime\prime}>0$ such that $\norm{C_b R(i\omega,A_b)} \leq C^{\prime\prime}\sqrt{|\omega|+1}$ for all $\omega \in \mathbb{R}$.
\end{lemma}
\begin{proof}
 First let us prove that $\norm{C_r R(i\omega,A_r)}$, where $A_r=\mathcal{A}_r|_{\mathcal{N}(\mathcal{B}_r)}$, $C_r=\mathcal{C}_r|_{\mathcal{N}(\mathcal{B}_r)}$, grows at most at a rate of $\sqrt{|\omega|},\ \omega\in\mathbb{R}$. Let us write $A_r$ as bounded perturbation of a skew-adjoint operator. i.e., $A_r=A_0+B_0$ where $A_0$ and $B_0$ are given as in Lemma \ref{lem.observability} and $A_0^*=-A_0,\ D(A_0^*)=D(A_0)$ and $B_0^*=B_0$. 
Now, for the system $(A_0,B_r,C_r)$, using duality between $D(A_0^*)$ and $X_{r_{-1}}$ (see \cite[Sec. 2.10]{TucWei2009}), we have $B_r^*\in \mathcal{L}(D(A_0^*),U_r)$ is the adjoint of $B_r \in \mathcal{L}(U_r,X_{r_{-1}})$ in the sense that 
\begin{align*}
\left\langle x_r,B_ru_r \right\rangle_{D(A_0^*),X_{r_{-1}}}= \left\langle B_r^*x_r,u_r \right\rangle_{U_r},\quad x_r\in D(A_0^*),\ u_r\in U_r
\end{align*} and $A_{0_{-1}}$ is the adjoint of $A_0^*$ in the sense that
\begin{align*}
\left\langle \psi_r,A_{0_{-1}}x_r \right\rangle_{D(A_0^*),X_{r_{-1}}}= \left\langle A_0^*\psi_r,x_r \right\rangle_{X_r},\quad \psi_r\in D(A_0^*),\ x_r\in X_r.
\end{align*}
 Moreover, using \cite[Rem.10.1.6]{TucWei2009}, we have
\begin{align*}
\left\langle \mathcal{B}_rx_r,B_r^*\psi_r \right\rangle_{U_r}=\left\langle \mathcal{A}_0x_r,\psi_r \right\rangle_{X_r}-\left\langle x_r,A_0^*\psi_r \right\rangle_{X_r},\quad \psi_r\in D(A_0^*),\ x_r\in D(\mathcal{A}_0)
\end{align*} and by direct computation using integration by parts we obtain $B_r^*x_r=C_rx_r$ for $x_r \in D(A_0)$. 
Therefore, for all $x_r \in D(A_0)$, $u_r \in U_r$ and $i\omega\in \rho(A_0)\cap i\mathbb{R}$, we have 
\begin{align*}
\left\langle x_r,R(i\omega,A_{0_{-1}})B_ru_r\right\rangle_{X_r}&= \left\langle R(\overline{i\omega},A_0^*)x_r,B_ru_r\right\rangle_{D(A_0^*),X_{r_{-1}}}\\
                                                               &=\left\langle B_r^*R(\overline{i\omega},A_0^*)x_r,u_r\right\rangle_{U_r},\\
																															&=-\left\langle C_rR(i\omega,A_0)x_r,u_r \right\rangle_{U_r}. 
\end{align*}
Since $A_r=A_0+B_0$ and $i\mathbb{R}\subset \rho(A_r)$, for $i\omega \in \rho(A_0)\cap i\mathbb{R}$, we obtain
\begin{align*}
&\left\langle x_r,R(i\omega,A_{r_{-1}})B_ru_r\right\rangle_{X_r}\\
&\qquad=\left\langle x_r,(I-R(i\omega,A_{0_{-1}})B_0)^{-1}R(i\omega,A_{0_{-1}})B_ru_r\right\rangle_{X_r},\\
																															 &\qquad=\left\langle (I+B_0R(i\omega,A_0))^{-1}x_r,R(i\omega,A_{0_{-1}})B_ru_r\right\rangle_{X_r},\\
																															 &\qquad=-\left\langle C_r R(i\omega,A_0)(I+B_0R(i\omega,A_0))^{-1} x_r,u_r\right\rangle_{U_r},\\
																 																 &\qquad = -\left\langle C_rR(i\omega,A_r)(I+2B_0R(i\omega,A_r))^{-1}x_r,u_r\right\rangle_{U_r}.
\end{align*}
Since $x_r \in X_r$ and $u_r \in U_r$ are arbitrary, we have \begin{equation}\label{resImax}C_rR(i\omega,A_r)=-(R(i\omega,A_{r_{-1}})B_r)^*(I+2B_0R(i\omega,A_r)),\ i\omega \in \rho(A_0)\cap i\mathbb{R}\end{equation}
where using Lemma \ref{lem.observability}, we have that  $\sup_{\omega \in \mathbb{R}}\norm{I+2B_0R(i\omega,A_r)}<\infty$. Since $A_0$ has discrete spectrum, the continuity of $R(i\omega,A_r)$, $C_rR(i\omega,A_r)$ and $R(i\omega,A_{r_{-1}})B_r$ with respect to $i\omega$ on $i\mathbb{R}$ imply that (\ref{resImax}) holds for all $i\omega \in i\mathbb{R}$. 
Now, using Lemma \ref{corresolbd1}, we have that there exists $C_0>0$ such that $\norm{C_rR(i\omega,A_r)}\leq C_0 \sqrt{|\omega|+1},\ \omega\in\mathbb{R}$. We can analogously show that there exists $C_0'>0$ such that $\norm{C_lR(i\omega,A_l)}\leq C_0' \sqrt{|\omega|+1},\ \omega\in\mathbb{R}$. Thus $\norm{C_bR(i\omega,A_b)}\lesssim \sqrt{|\omega|+1},\ \omega\in\mathbb{R}$.
\end{proof}
\begin{lemma}\label{ub_PbPc}
Let $P_b(\cdot)$ and $P_c(\cdot)$ be the transfer functions of the beam system $(\mathcal{A}_b,\mathcal{B}_b,\mathcal{C}_b)$ and the rigid body $(A_c,B_c,C_c)$, respectively. Then there exist $\omega_0, \tilde{M}>0$ such that $\norm{(I+P_b(i\omega)P_c(i\omega))^{-1}}\leq \tilde{M}$ for all $|\omega|\geq \omega_0$. 
\end{lemma}
\begin{proof}
From equation \eqref{beamtf1} in the proof of Lemma \ref{lemtf} and from equation \eqref{tfrigidbody} in the proof of Lemma \ref{corresolbd1}, we have
\begin{align*}
I+P_b(i\omega)P_c(i\omega) = \begin{bmatrix}Q_1(\omega) & 0 \\ 0 & Q_2(\omega)\end{bmatrix},\quad \omega \in \mathbb{R}\backslash \{0\},               
\end{align*}
where
\begin{align*}
Q_{1}(\omega)&=1-\frac{4EI}{m} \frac{{\alpha(\omega)}^3}{\omega^2} C_{2,\omega},\
Q_{2}(\omega)=1-\frac{4EI}{I_m} \frac{{\alpha(\omega)}}{\omega^2} C_{3,\omega},\\
\alpha(\omega)  &= |\alpha(\omega)|\bigg(\cos{\bigg(\frac{\theta(\omega)+2\pi k}{4}\bigg)}+i \sin{\bigg(\frac{\theta(\omega)+2\pi k}{4}\bigg)}\bigg),\ k=0,1,2,3,\\
|\alpha(\omega)|&=\bigg(\frac{\rho a}{EI}|\omega|\sqrt{\omega^2+\gamma^2(\rho a)^{-2}}\bigg)^{\frac{1}{4}},\
\theta(\omega) = \tan^{-1}(\frac{-\gamma (\rho a)^{-1}}{\omega}),
\end{align*}
and $C_{2,\omega}$ and $C_{3,\omega}$ are defined in \eqref{eq.C2C3}. 
We will show that there exist $\omega_0>0$ and $c_1, c_2>0$ such that $|Q_1(\omega)|>c_2$ and $|Q_2(\omega)|>c_1$ for all $|\omega|\geq \omega_0$.
Since $\abs{C_{2,\omega}}$ and $\abs{C_{3,\omega}}$ have similar terms for all the four roots of $\alpha(\omega)$, we restrict our analysis to the principal branch of the fourth root of $\alpha(\omega)$ and analogous arguments can be used to show that the statement is also valid for the other roots of $\alpha(\omega)$.
 
We have from equation \eqref{eq.C3} that there exists $M_1,\omega_0>0$ such that $|C_{3,\omega}|\leq M_1 \sqrt{|\omega|}$ for all $|\omega|\geq \omega_0$. Therefore, for $|\omega|\geq \omega_0$,  we have that
\begin{align*}
\abs{Q_2(\omega)-1}=\bigg|\frac{4EI}{I_m} \frac{{\alpha(\omega)}}{\omega^2} C_{3,\omega}\bigg| \lesssim 4 \frac{EI}{I_m}(\frac{\rho a}{EI})^{\frac{1}{4}}\frac{1}{|\omega|} \to 0 
\end{align*}
as $|\omega|\to \infty$.
This implies that there exists $c_1>0$ such that $|Q_{2}(\omega)|>c_1$ for all $|\omega|\geq \omega_0$. 

Now it remains to show that there exists $c_2>0$ such that $\abs{Q_{1}(\omega)}\geq c_2$ for all $|\omega|\geq \omega_0$. We begin by showing that 
 if we define $f(\omega)=\frac{2EI\ga(\gw)^3}{m\gw^2}$ and $\tilde{Q}_1(\gw)=1+f(\gw)\tan(\ga(\gw))$, then
\eqn{
\label{Q1Q1tildeconvergence}
\lim_{\abs{\gw}\to\infty} \abs{Q_1(\gw)- \tilde{Q}_1(\gw)}=0.
}
This will imply that $\abs{Q_1(\gw)}$ is uniformly bounded from below for $|\omega|\geq \omega_0$ if and only if the same is true for  $\abs{\tilde{Q}_1(\gw)}$.
We have from equation \eqref{eq.C2C3} in Lemma \ref{lemtf} that
\begin{align*}
C_{2,\omega}&=\frac{\cos{(\alpha(\omega))}}{\sinh{(\alpha(\omega))}+\sin{(\alpha(\omega))}}- \frac{1}{2} (1+C_{5,\omega}) C_{6,\gw}\\
C_{5,\omega}&:=
\frac{\sinh{(\alpha(\omega))}\sin{(\alpha(\omega))}}{1+\cosh{(\alpha(\omega))}\cos{(\alpha(\omega))}},\ 
C_{6,\omega}:=
\frac{\cosh{(\alpha(\omega))}+\cos{(\alpha(\omega))}}{\sinh{(\alpha(\omega))}+\sin{(\alpha(\omega))}}.
\end{align*} 
We have from equations \eqref{ub2} and \eqref{ub3} that $|\cos{(\alpha(\omega))}/(\sinh{(\alpha(\omega))}+\sin{(\alpha(\omega))})|$ and $\abs{C_{6,\gw}}$ 
are uniformly bounded for $|\omega|\geq \omega_0$.
Thus for all $\abs{\gw}\geq \omega_0$, we have
\eq{
 \abs{Q_1(\gw)- \tilde{Q}_1(\gw)}
 &=\abs{2f(\gw)C_{2,\gw} + f(\gw)\tan(\ga(\gw))}\\
&\lesssim \abs{f(\gw)}
+\abs{f(\gw)(C_{5,\gw}C_{6,\gw}-\tan(\ga(\gw)))}\\
&\lesssim \abs{f(\gw)}
 + \abs{f(\gw)C_{5,\gw}(C_{6,\gw}-1)}\\
&\quad +\abs{f(\gw)\tan(\ga(\gw))(\tanh(\ga(\gw))-1)}\\
&\quad
+\abs{f(\gw)(C_{5,\gw}-\tanh(\ga(\gw))\tan(\ga(\gw)))}.
}
Using the definition of $\alpha(\omega)$, it is straightforward to show that $C_{6,\gw}\to 1$ and $\tanh(\ga(\gw))\to 1$ as $\abs{\gw}\to\infty$.
Moreover, as shown in \eqref{eq.bd} and \eqref{ub1}, we have $\abs{C_{5,\gw}}\lesssim \sqrt{|\gw|}$ and $\abs{\tan(\ga(\gw))}\leq \abs{\coth{(y_{\omega})}}+\abs{\tanh{(y_{\omega})}}\lesssim \sqrt{|\gw|}$ for $\abs{\gw}\geq \omega_0$.
Because of this, $\abs{f(\gw)C_{5,\gw}}$ and $\abs{f(\gw)\tan(\ga(\gw))}$ are uniformly bounded for $\abs{\gw} \geq \omega_0$, and therefore $\abs{f(\gw)C_{5,\gw}(C_{6,\gw}-1)}\to 0$ and $\abs{f(\gw)\tan(\ga(\gw))(\tanh(\ga(\gw))-1)}\to 0$ as $\abs{\gw}\to\infty$.
Finally, the last term in the estimate for 
 $\abs{Q_1(\gw)- \tilde{Q}_1(\gw)}$ satisfies
\eq{
&\abs{f(\gw)(C_{5,\gw}-\tanh(\ga(\gw))\tan(\ga(\gw)))}\\
&\quad =\abs{f(\gw)}
\bigg|
\frac{\sinh{(\alpha(\omega))}\sin{(\alpha(\omega))}}{1+\cosh{(\alpha(\omega))}\cos{(\alpha(\omega))}}
-\frac{\sinh{(\alpha(\omega))}\sin{(\alpha(\omega))}}{\cosh{(\alpha(\omega))}\cos{(\alpha(\omega))}}
\bigg|\\
&\quad =\abs{f(\gw)\tanh(\ga(\gw))\tan(\ga(\gw))}
\bigg|
\frac{1}{1+\cosh{(\alpha(\omega))}\cos{(\alpha(\omega))}}
\bigg|\to 0 
}
as $\abs{\gw}\to \infty$, since $\abs{f(\gw)\tanh(\ga(\gw))\tan(\ga(\gw))}$ is uniformly bounded for $\abs{\gw}\geq \omega_0$, and 
$\abs{\cosh(\alpha(\omega))\cos(\alpha(\omega))}\to \infty$ as $\abs{\gw}\to \infty$.
This finally shows that~\eqref{Q1Q1tildeconvergence} holds.

We claim that there exists $c'>0$  such that $\abs{\tilde{Q}_{1}(\omega)}\geq c'$ for all $\gw \geq \gw_0$. 
The case where $\omega$ is negative can be proved analogously. We will use proof by contradiction. To this end we assume that no such $c'>0$ exists. This implies that there exists a sequence $(\gw_k)_{k}\subset \R_+$ such that $\gw_k\to\infty$ as $k\to\infty$ and $\abs{\tilde{Q}_{1}(\omega_k)}\to 0$ as $k\to\infty$. 
 Separating real and imaginary parts of $\tilde{Q}_{1}(\omega_k)$ and denoting
$x_k = \re \alpha (\gw_k)$, $y_k = \im \alpha (\gw_k)$, $R_{1,k}=\re(f(\gw_k))\sin x_k$, $R_{2,k}=\re(f(\gw_k))\cosh y_k$, $I_{1,k}=\im(f(\gw_k))\cosh y_k$, $I_{2,k}=\im(f(\gw_k))\sin x_k$, we obtain
 
\eq{
\tilde{Q}_{1}(\omega_k) = 1+ \frac{R_{1,k}\cos x_k- I_{1,k}\sinh y_k}{\cos^2 x_k+ \sinh^2 y_k}
+ i \frac{R_{2,k}\sinh y_k + I_{2,k}\cosh x_k}{\cos^2 x_k+\sinh^2 y_k}.
}
Since we consider the principal branch of the fourth root of $\alpha(\omega_k)$, we have that there exist $m_1, m_2, m_3, m_4>0$ and $N_1\in \mathbb{N}$ such that 
\begin{align*}
m_1\sqrt{\omega_k}&\leq |x_k| \leq m_2 \sqrt{\omega_k}\\
\frac{m_3}{\sqrt{\omega_k}} &\leq |y_k| \leq \frac{m_4}{\sqrt{\omega_k}}
\end{align*}
for all $k\geq N_1$. This implies that there exist $m_5,m_6>0$ and $N_2\geq N_1$ such that $m_5\omega_k^{-1/2} \leq |\sinh y_k| \leq m_6\omega_k^{-1/2}$ for all $k\geq N_2$. Since $ y_k\to 0$, we have $\cosh y_k\to 1$ as $k\to \infty$, and thus there exist $m_7, m_8, m_9, m_{10}>0$ and $N_3\geq N_2$ such that $m_7\omega_k^{-1/2}\leq \abs{R_{2,k}}\leq m_8\omega_k^{-1/2}$ and $m_9\omega_k^{-3/2}\leq \abs{I_{1,k}}\leq m_{10}\omega_k^{-3/2}$ for all $k\geq N_3$. 

We will first show that $\abs{\cos x_k} \to 0$ as $k\to \infty$. Indeed, we have 
\eq{
\abs{R_{1,k}\cos x_k - I_{1,k}\sinh y_k }
&\leq 
\abs{\re (f(\gw_k))} + \abs{\im(f(\gw_k))}\abs{\sinh y_k} \abs{\cosh y_k}\\ 
&\lesssim \frac{1}{\sqrt{\gw_k}} 
}
for all $k\geq N_3$ 
and since the assumption $\abs{\tilde{Q}_{1}(\omega_k)}\to 0$ implies $\re \tilde{Q}_{1}(\omega_k)\to 0$, we must have $\cos^2 x_k+ \sinh^2 y_k\to 0$ as $k\to \infty$.
Thus $\abs{\cos x_k}\to 0$ as $k\to\infty$, and consequently also $\abs{\sin x_k}\to 1$ as $k\to\infty$. This further implies that there exist $m_{11}, m_{12}, m_{13}, m_{14}>0$ and $N_4\geq N_3$ such that $m_{11}\omega_k^{-1/2}\leq\abs{R_{1,k}}\leq m_{12}\omega_k^{-1/2}$ and $m_{13}\omega_k^{-3/2}\leq\abs{I_{2,k}}\leq m_{14}\omega_k^{-3/2}$ for all $k \geq N_4$. 
We consider the following cases.

 \textbf{Case 1 (fast decay of $\abs{\cos x_k}$):}
Consider the subsequence of  $(\gw_k)$ consisting of those elements $\gw_k$ which satisfy $\abs{\cos x_k}\leq 1/\gw_k$.
Then we have 
\eq{
\bigg|\frac{R_{1,k}\cos x_k- I_{1,k}\sinh y_k}{\cos^2 x_k+ \sinh^2 y_k}\bigg|
&\leq 
\frac{\abs{R_{1,k}\cos x_k- I_{1,k}\sinh y_k}}{\sinh^2 y_k}\\
&\lesssim \frac{|\cos x_k|/\sqrt{\gw_k}+ 1/\gw_k^2}{1/\gw_k}
\lesssim \frac{1}{\sqrt{\gw_k} } + \frac{1}{\gw_k} 
}
for all $k\geq N_4$. However, this implies $\tilde{Q}_{1}(\omega_k)\not\to 0$ as $k\to\infty$, since $\re \tilde{Q}_{1}(\omega_k)\to 1$. This implies that the subsequence of $(\gw_k)_k$ consisting of elements such that $\abs{\cos x_k}\leq 1/\gw_k$ must have at most finite number of elements.

\textbf{Case 2 (slow decay of $\abs{\cos x_k}$):}
As shown above, we necessarily have there exist $N_5\geq N_4$ such that $\abs{\cos x_k}>1/\gw_k$ for all $k\geq N_5$, and we will now restrict our attention to this range of the indices $k$. 
Then 
\eq{
\bigg|\frac{R_{1,k}\cos x_k - I_{1,k}\sinh y_k}{R_{1,k}\cos x_k}\bigg| \to 1
\ \mbox{and}\ 
\bigg|\frac{R_{2,k}\sinh y_k + I_{2,k}\cos x_k}{R_{2,k}\sinh y_k}\bigg|\to 1 
} as $k \to \infty$.
In addition, for $k\geq N_5$, we have
\eq{
\abs{\im (\tilde{Q}_{1}(w_k))} 
&= \frac{\abs{R_{2,k} \sinh y_k + I_{2,k} \cos x_k}}{\cos^2x_k + \sinh^2 y_k}\\
&= \bigg|\frac{R_{2,k} \sinh y_k + I_{2,k} \cos x_k}{R_{2,k}\sinh y_k} \bigg|\cdot \frac{\abs{R_{2,k} \sinh y_k }}{\cos^2x_k + \sinh^2 y_k}\\
&\gtrsim 
  \frac{1/\gw_k }{\cos^2x_k + \sinh^2 y_k}
}
and
\eq{
\abs{\re(\tilde{Q}_{1}(w_k))-1}
&=\frac{\abs{R_{1,k} \cos x_k - I_{1,k} \sinh y_k}}{\cos^2x_k + \sinh^2 y_k}\\
&=
\bigg|\frac{R_{1,k} \cos x_k - I_{1,k} \sinh y_k}{R_{1,k}\cos x_k}\bigg|
\cdot\frac{\abs{R_{1,k} \cos x_k} }{\cos^2x_k + \sinh^2 y_k}\\
& \gtrsim \frac{ \abs{\cos x_k}/\sqrt{\gw_k} }{\cos^2x_k + \sinh^2 y_k} .
}
Using these estimates, we have that
\eq{
\frac{1}{\abs{\cos x_k}\sqrt{\gw_k}}
= \frac{1/\gw_k}{\cos^2 x_k+ \sinh^2 y_k}\cdot \frac{\cos^2 x_k+ \sinh^2 y_k}{\abs{\cos x_k}/\sqrt{\gw_k}} \to 0
}
as $k\to \infty$ since $\abs{\re\tilde{Q}_{1}(\omega_k)-1}\to 1$ and $\abs{\im \tilde{Q}_{1}(\omega_k)}\to 0$ as $k\to \infty$ .
Because of this we also have
\eq{
\bigg|\frac{\cos^2 x_k+\sinh^2 y_k}{\cos^2 x_k}-1\bigg|
= \frac{\sinh^2 y_k}{\cos^2 x_k}
\lesssim \frac{1}{(\sqrt{\gw_k}\cos x_k)^2}
\to 0
}
as $k\to \infty$. Finally, using this property we have that for all $k\geq N_5$
\eq{
\bigg|\frac{R_{1,k} \cos x_k - I_{1,k} \sinh y_k}{\cos^2x_k + \sinh^2 y_k}\bigg|
&= 
\bigg|\frac{\cos^2 x_k}{\cos^2 x_k+\sinh^2 y_k}\bigg|\cdot\bigg|\frac{R_{1,k} \cos x_k - I_{1,k} \sinh y_k}{\cos^2x_k}\bigg|\\
&\lesssim \frac{1}{\abs{\cos x_k}\sqrt{\gw_k}}  +\frac{ 1/\gw_k}{(\cos x_k \sqrt{\gw_k})^2}
}
decays to zero as $k\to\infty$.
However, this implies that $\re \tilde{Q}_{1}(\omega_k)\to 1 \neq 0$ as $k\to \infty$ which contradicts the assumption that $\abs{\tilde{Q}_{1}(\omega_k)}\to 0$ as $k\to \infty$. Hence there exists $c'>0$ such that $\abs{\tilde{Q}_{1}(\omega)}\geq c'$ for all $\gw\geq \gw_0$.

Finally, we have that there exist $\omega_0,c_1,c_2>0$ such that $|Q_1(\omega)|>c_2$ and $|Q_2(\omega)|>c_1$ for all $|\omega|\geq \omega_0$. This implies that $\norm{(I+P_b(i\omega)P_c(i\omega))^{-1}}^2 \leq \frac{1}{c_2^2}+\frac{1}{c_1^2}$ for all $|\omega|\geq \omega_0$, which completes the proof. 
\end{proof}
Having the above results, now we are ready to prove the main theorem.

\begin{proof} [Proof of Theorem \ref{mainthm}]
From Lemmas \ref{lemtf} and \ref{corresolbd1}, we have that $P_b(0)$ and $I+P_b(i\omega)P_c(i\omega),\ \omega \in \mathbb{R}\backslash \{0\}$ are nonsingular. These properties in Lemma \ref{beamres} imply that the resolvent $R(i\omega,A)$ exists for all $\omega \in \mathbb{R}$ and is given by the equations (\ref{reseqn}), (\ref{subreseqn}) and \eqref{schcom}. Therefore
\begin{equation}\label{resnorm}
\begin{aligned}
\norm{R(i\omega,A)}^2 &\leq \norm{R(i\omega,A_b)-R(i\omega,A_{b_{-1}})B_bC_c S(i\omega)B_c C_b R(i\omega,A_b)}^2 \\
                      &+\norm{R(i\omega,A_{b_{-1}})B_bC_c S(i\omega)}^2+\norm{S(i\omega)B_c C_bR(i\omega,A_b)}^2+\norm{S(i\omega)}^2
\end{aligned}
\end{equation}
where
$
S(i\omega) = \frac{1}{i\omega}+\frac{1}{\omega^2} B_c P_b(i\omega) (I+P_b(i\omega)P_c(i\omega))^{-1} C_c	
$
for $\omega \in \mathbb{R}\backslash\{0\}$ and $S(0)=(B_c P_b(0)C_c)^{-1}$.

From Lemma \ref{lemtf}, we have that there exists $M>0$ such that $\|P_b(i\omega)\|\leq M(|\omega|+1)$ for all $\omega \in \mathbb{R}$. 
From Lemma \ref{ub_PbPc}, we have that there exist $\omega_0, \tilde{M}>0$ such that $\norm{(I+P_b(i\omega)P_c(i\omega))^{-1}}\leq \tilde{M}$ for all $|\omega|\geq \omega_0$. Moreover, from Lemma \ref{lem.observability}, we have that $\|R(i\omega,A_b)\|$ is uniformly bounded and from Lemmas \ref{corresolbd1} and \ref{corresolbd2}, we have that there exist $C',C^{\prime\prime}>0$ such that $\norm{R(i\omega,A_{b_{-1}})B_b}\leq C' \sqrt{|\omega|+1}$ and $\norm{C_bR(i\omega,A_b)}\leq C^{\prime\prime} \sqrt{|\omega|+1}$ for all $\omega\in\mathbb{R}$. These estimates imply that there exists $ M_1>0$ such that $\norm{S(i\omega)}\leq M_1$ for all $|\omega|\geq \omega_0$ and this further from equation \eqref{resnorm} implies that there exists $M_0>0$ such that $\norm{R(i\omega,A)}\leq M_0$ for all $|\omega|\geq \omega_0$. Since from Lemma \ref{beamres} we have $i\mathbb{R}\subset \rho(A)$, we conclude  that $R(i\omega,A)$ is uniformly bounded, which completes the proof. 
\end{proof}

\section{Robust Output Regulation of the Satellite Model} \label{sec.OR}
In this section, we present two controllers that solve the robust output regulation problem for the satellite system. We start by formulating the robust output regulation problem followed by the controllers that achieve the robust output tracking of the given reference signals. In addition, we present simulation results demonstrating the effectiveness of the controllers.

From the previous sections, the satellite system with control and observations on the rigid body is given by,
\begin{equation}\label{eq:ss}
\begin{aligned}
\dot{x}(t)&=Ax(t)+Bu(t)+B_d w_d(t),\\
y(t) &= Cx(t).
\end{aligned}
\end{equation}
with $A=\begin{bmatrix} \mathcal{A}_b & 0\\ -B_c \mathcal{C}_b & 0 \end{bmatrix},\ D(A)=\{(x_b,x_c) \in D(\mathcal{A}_b)\times X_c\  :\  \mathcal{B}_bx_b= C_c x_c\},\ B=\begin{bmatrix} 0 \\ B_c\end{bmatrix},\ B_d=\begin{bmatrix} B_{d0} & 0 \\ 0 & B_c\end{bmatrix},\ C=\begin{bmatrix} 0 & C_c\end{bmatrix},\ x(t)=\begin{bmatrix} x_b(t) \\ x_c(t) \end{bmatrix}$.
Here the operator $A$ generates an exponentially stable semigroup. 

The reference signals to be tracked and the disturbance signals to be rejected are of the form
\begin{equation}\label{refsig}
y_{ref}(t) = a_0+\sum_{k=1}^q{[a_k \cos(\omega_k t)+b_k \sin(\omega_k t)]},
\end{equation}
\begin{equation}\label{dissig}
w_d(t) = c_0+\sum_{k=1}^q{[c_k \cos(\omega_k t)+d_k \sin(\omega_k t)]},
\end{equation}
where $0<\omega_1<\omega_2<\cdots<\omega_q$ are known frequencies and $\{a_k\}_{k=0}^q, \{b_k\}_{k=1}^q$, $\{c_k\}_{k=0}^q,\ \{d_k\}_{k=1}^q$ are possibly unknown constant coefficients. 

We construct a dynamic error feedback controller of the form 
\begin{equation}\label{eq:contr}
\begin{aligned}
\dot{z}(t) &= \mathcal{G}_1 z(t)+\mathcal{G}_2 e(t),\quad z(0) = z_0,\\
u(t) &= K z(t)-\kappa e(t),
\end{aligned}
\end{equation}
on a Hilbert space $Z$, where $e(t)=y(t)-y_{ref}(t)$ is the regulation error, $y_{ref}(t)$ a given reference signal, $\mathcal{G}_1:D(\mathcal{G}_1)\subset Z \rightarrow Z$ generates a strongly continuous semigroup on $Z$, $\mathcal{G}_2 \in \mathcal{L}(\mathbb{R}^2,Z)$, $K \in \mathcal{L}(Z,\mathbb{R}^2)$ and $\kappa \in \mathbb{R}^{2\times 2}$, such that robust output regulation of the satellite system is achieved with a suitable choice of the parameters $(\mathcal{G}_1,\mathcal{G}_2,K,\kappa)$. 

Let us denote $X_e=X\times Z$ to be the extended state space and $x_e(t)=(x(t),z(t))^T$ be the extended state. Then the closed-loop system containing the satellite system (\ref{eq:ss}) and the controller (\ref{eq:contr}) is given by
\begin{equation}\label{eq:CLS}
\begin{aligned}
\dot{x}_e(t)&=A_ex_e(t)+B_e u_e(t),\quad x_e(0)=x_{e0},\\
e(t)&= C_ex_e(t)+D_e u_e(t),
\end{aligned}
\end{equation}
where $A_e=\begin{bmatrix}A-B\kappa C & BK\\ \mathcal{G}_2C & \mathcal{G}_1 \end{bmatrix}$, $B_e=\begin{bmatrix} B_d & B\kappa\\ 0 & -\mathcal{G}_2 \end{bmatrix}$, $C_e=\begin{bmatrix} C & 0\end{bmatrix}$, $D_e=\begin{bmatrix} 0 & -I_Y \end{bmatrix}$ and $u_e(t)=\begin{bmatrix} w_d(t) \\ y_{ref}(t) \end{bmatrix}$. The operator $A_e$ generates a strongly continuous semigroup $T_e(t)$ on $X_e$.

\textbf{The Robust Output Regulation Problem}. Choose the controller parameters $(\mathcal{G}_1,\mathcal{G}_2,K,\kappa)$ in such a way that
\begin{itemize}
	\item [(a)] The closed-loop semigroup $T_{e}(t)$ generated by $A_e$ is exponentially stable.
	\item [(b)] There exists $\alpha_1>0$ such that for all initial states $x_{e0}\in X_e$, for the reference signal of the form \eqref{refsig} and for the disturbance signal of the form \eqref{dissig}, the regulation error $e(t)$ satisfies 
	\begin{align*}
	e^{\alpha_1 t}\|y(t)-y_{ref}(t)\| \rightarrow 0\quad \text{as}\  t \rightarrow \infty.
	\end{align*}
	\item [(c)] If the operators $(\mathcal{A}_b,\mathcal{B}_b,\mathcal{C}_b,A_c,B_c,C_c)$ are perturbed in such a way that the perturbed closed-loop system is exponentially stable, the perturbed $(\mathcal{A}_b,\mathcal{B}_b,\mathcal{C}_b)$ is an impedance passive boundary control system and the perturbed $(A_c,B_c,C_c)$ is an impedance passive systems, then (b) continues to hold for some $\tilde{\alpha}_1>0$. 
\end{itemize}
\begin{remark}In the above, $\alpha_1$ and $\tilde{\alpha}_1$ are determined by the stability margins of the closed-loop system and the perturbed closed-loop system, respectively.\end{remark}

Next, we show that the transfer function $P(i\omega)$ of the satellite system is nonsingular for all $\omega\in\mathbb{R}$. Because of this, we can track signals containing components at all frequencies $\omega$.
\begin{lemma}
On the imaginary axis, the transfer function of the satellite system \eqref{eq:ss} has the form $P(i\omega)=C_cS(i\omega)B_c$ and it is nonsingular for all $\omega\in\mathbb{R}$.
\end{lemma}
\begin{proof}
The transfer function of \eqref{eq:ss} on the imaginary axis is given by $P(i\omega)=CR(i\omega,A)B$, where $R(i\omega,A)$ is the resolvent in \eqref{reseqn} of $A$. Replacing the operators by their expressions, we obtain
\begin{align*}
P(i\omega) &= \begin{bmatrix} 0 & C_c\end{bmatrix}\begin{bmatrix} R_{11}(i\omega) & R_{12}(i\omega) \\ R_{21}(i\omega) & R_{22}(i\omega)\end{bmatrix}\begin{bmatrix} 0 \\ B_c\end{bmatrix}\\
           &=C_cR_{22}(i\omega)B_c\\
					&=C_cS(i\omega)B_c.
\end{align*}
Since $B_c$, $C_c$ and $S(i\omega)$, $\omega\in\mathbb{R}$ are nonsingular, we have that $P(i\omega)$ is nonsingular for all $\omega \in \mathbb{R}$.
\end{proof}

\subsection{Robust Controllers for the Satellite System} \label{sec.controllers}

In this section, we present two internal model based controllers for the robust output regulation of the satellite system. 
\subsubsection{A Passive Controller for the Satellite Model}\label{sec.passive}
We have that the satellite system is (\ref{eq:ss}) an impedance passive system 
and exponentially stable. Therefore, based on \cite[Thm. 1.2]{RebarberWeiss2003} and \cite[Def. 5.1]{Lassi2019}, we can construct a passive controller for the robust output tracking of the given sinusoidal reference signals. We choose $Z=(\mathbb{R}^2)^{2q+1}$,
\begin{equation}\label{passivecr}
\begin{aligned}
\mathcal{G}_1 &=\text{diag}(G_0, G_1 , G_2 ,\cdots,G_q),\\
G_0 &=0_Y,\ 
G_k = \begin{bmatrix}0 & \omega_kI_Y \\ -\omega_k I_Y & 0\end{bmatrix},\quad k=1,2,\cdots, q,\\
\mathcal{G}_2 &=(\mathcal{G}_2^k)_{k=0}^q,\ \mathcal{G}_2^0=-I_Y,\ \mathcal{G}_2^k=-c_1\begin{bmatrix}I_Y \\ 0\end{bmatrix}, \ k=1,2,\cdots,q, \\
K &=- \mathcal{G}_2^*, \ \text{and}\  
\kappa =c_2 I_Y,
\end{aligned}
\end{equation}
where $c_1,c_2>0$ affect the stability properties of the closed-loop system.
\begin{theorem}
The controller \eqref{eq:contr} with the choices of parameters in \eqref{passivecr} solves the robust output regulation problem for the satellite model.
\end{theorem}
\begin{proof}
We have that the satellite system (\ref{eq:ss}) is impedance passive and exponentially stable and the choices of parameters in \eqref{passivecr} are adopted from \cite[Def. 5.1]{Lassi2019}. Therefore, by \cite[Thm. 5.2]{Lassi2019}, the controller \eqref{eq:contr}, \eqref{passivecr} solves the robust output regulation problem.  
\end{proof}
We note that the controller \eqref{eq:contr}, \eqref{passivecr} is the one given in \cite[Thm. 1.2]{RebarberWeiss2003} when $c_1$ and $c_2$ are chosen such that \eqref{eq:contr}, \eqref{passivecr} is a minimal realization of 
\begin{equation}\label{minrea}
C(s)=-C_0-\sum_{k=-q}^q{\frac{I_Y}{s-i\omega_k}},
\end{equation}
where $C_0\geq \frac{1}{2}I_Y$ and $\omega_{-k}=-\omega_k$. The assumption $\re P(i\omega_k)$ is nonsingular  for all $k=0,1,2,\cdots q$ in \cite[Thm. 1.2]{RebarberWeiss2003} can be relaxed due to the fact that the feedthrough operator $\kappa$ of the controller satisfies $\kappa>0$ (see \cite[sec. 5]{Lassi2019} for more details). 
\subsubsection{An Observer Based Controller for the Satellite Model} \label{sec.obs-con}

Since the input operator $B$ and the output operator $C$ are bounded, we can construct an observer based controller based on \cite{HamPoh2010} and \cite[Sec. VI]{Pau2016} for robust output tracking of the satellite system as follows.

We choose the state space of the controller as $Z = Z_0\times X$, where $Z_0 = (\mathbb{R}^2)^{2q+1}$. The controller parameters $(\mathcal{G}_1,\mathcal{G}_2,K,\kappa)$ of the dynamic error feedback controller (\ref{eq:contr}) are given by,
\begin{align*}
\mathcal{G}_1&=\begin{bmatrix} G_1 & 0\\ BK_1 & A+BK_2\end{bmatrix},\ 
\mathcal{G}_2=\begin{bmatrix}G_2\\ 0 \end{bmatrix},\ 
K = \begin{bmatrix} K_1 & K_2 \end{bmatrix},\ \kappa=0,
\end{align*}
where $K_1 \in \mathcal{L}(Z_0,\mathbb{R}^2), K_2 \in \mathcal{L}(X,\mathbb{R}^2)$. The operators $(G_1,G_2)$ are defined as  
\begin{align*}
G_1 &= \text{diag}(i\omega_{-q} I_{Y},\cdots i\omega_0 I_{Y},\cdots,i\omega_q I_{Y}) \in \mathcal{L}(Z_0),\\ G_2 &=(G_2^k)_{k=-q}^{q} \in \mathcal{L}(\mathbb{R}^2,Z_0),\ G_2^k=I_Y,\ k=-q,\cdots,q.
\end{align*}
We define an operator $H \in \mathcal{L}(X, Z_0)$ by $H=(H_k)_{k=-q}^q$ which is the solution of the Sylvester equation $G_1H = HA+G_2C$ and 
$H_k$ can be obtained by solving the system
\begin{equation}\label{eq.sylsol}
H_k=G_2^k CR(i\omega_k,A).
\end{equation}
Then we define $B_1=HB=(G_2^k P(i\omega_k))_{k=-q}^q\in \mathcal{L}(\mathbb{R}^2,Z_0)$.
Finally, we choose $K_1\in \mathcal{L}(Z_0,\mathbb{R}^2)$ in such a way that $G_1+B_1K_1 \in \mathcal{L}(Z_0)$ is Hurwitz and we define $ K_2 = K_1H$.

With the above parameters, the controller (\ref{eq:contr}) can be written as,
\begin{align}
\dot{z}_1(t) &= G_1z_1(t)+G_2e(t), \label{eq:contr1}\\
\dot{z}_2(t) &=BK_1z_1(t)+(A+BK_2)z_2(t), \label{eq:contr2}\\
u(t)&=Kz(t).
\end{align}
Here $z_1(t)\in Z_0,\ z_2(t)\in X(=X_b\times X_c)$. Equation (\ref{eq:contr1}) is the servocompensator on the state space $Z_0$  which contains internal model and it is an ODE system by construction.
Equation (\ref{eq:contr2}) is an observer for the satellite system on the state space $X$ and is given by,
\begin{align*}
\dot{\hat{x}}_{l1}(\xi,t)&=-\gamma (\rho a)^{-1} \hat{x}_{l1}(\xi,t)-EI  \hat{x}_{l2}^{\prime\prime}(\xi,t),\ -1<\xi<0,\\ \dot{\hat{x}}_{l2}(\xi,t)&= (\rho a)^{-1}  \hat{x}_{l1}^{\prime\prime}(\xi,t),\ -1<\xi<0,\\
\dot{\hat{x}}_{r1}(\xi,t)&=-\gamma (\rho a)^{-1} \hat{x}_{r1}(\xi,t)-EI  \hat{x}_{r2}^{\prime\prime}(\xi,t),\ 0<\xi<1,\\ \dot{\hat{x}}_{r2}(\xi,t)&= (\rho a)^{-1}  \hat{x}_{r1}^{\prime\prime}(\xi,t),\ 0<\xi<1,\\
\dot{\hat{x}}_{c1}(t,0)&= EI \hat{x}_{l2}'(\xi,t)|_{\xi=0}-EI \hat{x}_{r2}'(\xi,t)|_{\xi=0}+u_1(t),\\
\dot{\hat{x}}_{c2}(t,0)&=-EI \hat{x}_{l2}(\xi,t)|_{\xi=0}+EI \hat{x}_{r2}(\xi,t)|_{\xi=0}+u_2(t),\\
\hat{x}_{r2}(1,t)&=\hat{x}_{r2}'(1,t)=0,\ \hat{x}_{l2}(-1,t)=\hat{x}_{l2}'(-1,t)=0,\\
\hat{x}_{l1}(0,t)&=\hat{x}_{r1}(0,t)=\hat{x}_{c1}(t),\
\hat{x}_{l1}'(0,t)=\hat{x}_{r1}'(0,t)=\hat{x}_{c2}(t),
\end{align*}
where $\hat{x}_{l1}(\xi,t),\hat{x}_{l2}(\xi,t),\hat{x}_{r1}(\xi,t),\hat{x}_{r2}(\xi,t),\hat{x}_{c1}(\xi,t)$ and $\hat{x}_{c2}(\xi,t)$ are the estimates of $x_{l1}(\xi,t)$, $x_{l2}(t)$, $x_{r1}(\xi,t)$, $x_{r2}(\xi,t)$, $x_{c1}(\xi,t)$ and $x_{c2}(\xi,t)$, respectively, and $z_2(t)$ is given by $z_2(t)=(\hat{x}_{l1}(\cdot,t),\hat{x}_{l2}(\cdot,t),\hat{x}_{r1}(\cdot,t),\hat{x}_{r2}(\cdot,t),\hat{x}_{c1}(\cdot,t),\\ \hat{x}_{c2}(\cdot,t))^T$. This shows that the controller (\ref{eq:contr}) is a PDE-ODE system.
\begin{theorem}
The controller $(\ref{eq:contr})$ with the above choices of parameters solves the robust output regulation problem for the satellite system \eqref{eq:ss}.
\end{theorem}
\begin{proof}
Since the construction of the controller $(\ref{eq:contr})$ with the above choices of parameters is adopted from \cite[Sec. 7]{HamPoh2010} and \cite[Sec. VI]{Pau2016}, based on \cite[Thm. 15]{Pau2016}, the controller solves the robust output regulation problem for the satellite system \eqref{eq:ss}. 
\end{proof}
\subsection{Robustness of Closed-loop Stability}In the case of the passive controller, the controller parameters $\mathcal{G}_2$, $K$ and $\kappa$ depend on the parameters $c_1$, $c_2$ and therefore the closed-loop stability margin $\alpha_1$ depends on the choice of the parameters $c_1$ and $c_2$. 
 On the other hand, for the observer based controller the closed-loop stability margin is determined by the minimum of stability margins of $A$ and $G_1+B_1K_1$, respectively, see the proof of \cite[Thm. 15]{Pau2016} for more details. The stability margin of $G_1+B_1K_1$ can be affected by adjusting the gain parameter $K_1$. This can be done for example by linear quadratic regulator design or pole placement.

From Section \ref{sec.controllers} we have that both controllers with suitable choices of parameters solve the robust output regulation problem. Therefore, $A_e$ generates an exponentially stable semigroup $T_e(t)$ and there exist $\alpha_1>0$ and $M_e\geq 1$ depending on the controller and the chosen parameters such that $\norm{T_e(t)}\leq M_e e^{-\alpha_1 t}$. If $\Delta \in \mathcal{L}(X_e)$ is a perturbation of $A_e$, where the perturbation is generated by the perturbations in $(\mathcal{A}_b,\mathcal{B}_b,\mathcal{C}_b,A_c,B_c,C_c)$, such that $\norm{\Delta}< \alpha_1/M_e$, then $A_e+\Delta$ generates an exponentially stable semigroup $\tilde{T}_e(t)$ on $X_e$ and $\norm{\tilde{T}_e(t)}\leq M_e e^{(-\alpha_1+M_e \norm{\Delta}) t}$ for all $t\geq 0$. Therefore the stability margin $\tilde{\alpha}_1$ of the perturbed semigroup $\tilde{T}_e(t)$ satisfies $\tilde{\alpha}_1\geq \alpha_1-M_e \norm{\Delta}$. In addition, the semigroup $T_e(t)$ may remain exponentially stable under perturbations with large norms in which cases the decay rates cannot be estimated explicitly by using the perturbation formula.
 
\subsection{Simulations}
Simulations are carried out in Matlab using passive and observer based controllers on the time interval $t=[0,15]$. We choose $m=1, I_m=1, E=1, I = 1, \rho=1, a=1$ and $\gamma=5$.  We track the reference signal $y_{ref}(t)=\begin{bmatrix}1+3\cos(t) & 2-\sin(5t)+1.5\cos(2t)\end{bmatrix}^T$ and reject the disturbance signal $w_d(t)=\begin{bmatrix} 0 & 0 & 10 & 15\end{bmatrix}^T$. Thus, the frequencies $\{\omega_k\}_{k=0}^q$ with $q=3$ are $\{0,1,2,5\}$. We choose the controller initial state as $z_0=0$ and the initial state for the satellite system as $x_0=\begin{bmatrix}0 & 4(1+\xi)^2 & 0 & 4(1-\xi)^2 & 0 & 0\end{bmatrix}^T$. The solutions of the satellite system are approximated using Legendre spectral Galerkin method \cite{Kri2020}. The number of basis functions used for the approximation is $N=10$. 

 The controller parameters of the passive controller are chosen as in Section \ref{sec.passive}. To maximize the stability margin, ranges of values of the parameters $c_1$ and $c_2$ were tested. The closed-loop stability margin $\alpha_1$ and $\int_0^{15}{\norm{e(t)}^2 dt}$ for different parameter values are plotted in Figures \ref{fig:c1} and \ref{fig:c2}, respectively. The figures indicate that smaller values of $c_1$ and $c_2$ result in larger closed-loop stability margin and larger transient errors. By choosing $c_1=2.5$ and $c_2=4$, the output tracking and the tracking errors are depicted in Figures \ref{fig.trackinglowgain} and \ref{fig.trackingerrors}, respectively.  

\begin{figure}%
    \centering
    \subfloat{{\includegraphics[width=5cm]{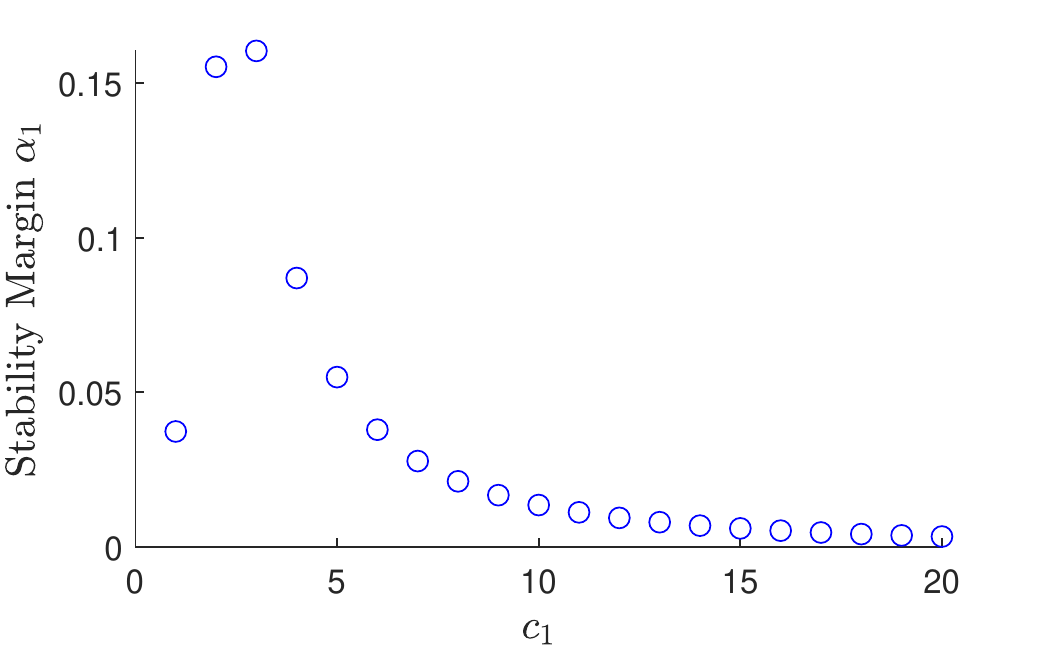} }}
    \qquad
    \subfloat{{\includegraphics[width=5cm]{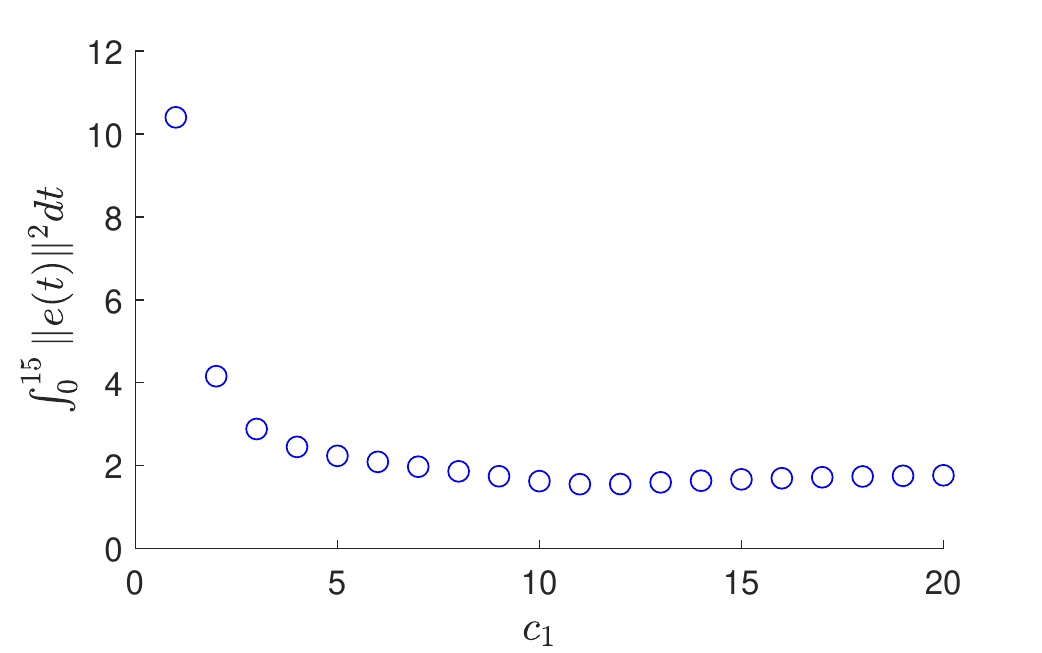} }}
    \caption{The closed-loop stability margin and $\int_0^{15}\norm{e(t)}^2 dt$ for the passive controller with $c_2=4$}
    \label{fig:c1}
\end{figure}
\begin{figure}%
    \centering
    \subfloat{{\includegraphics[width=5cm]{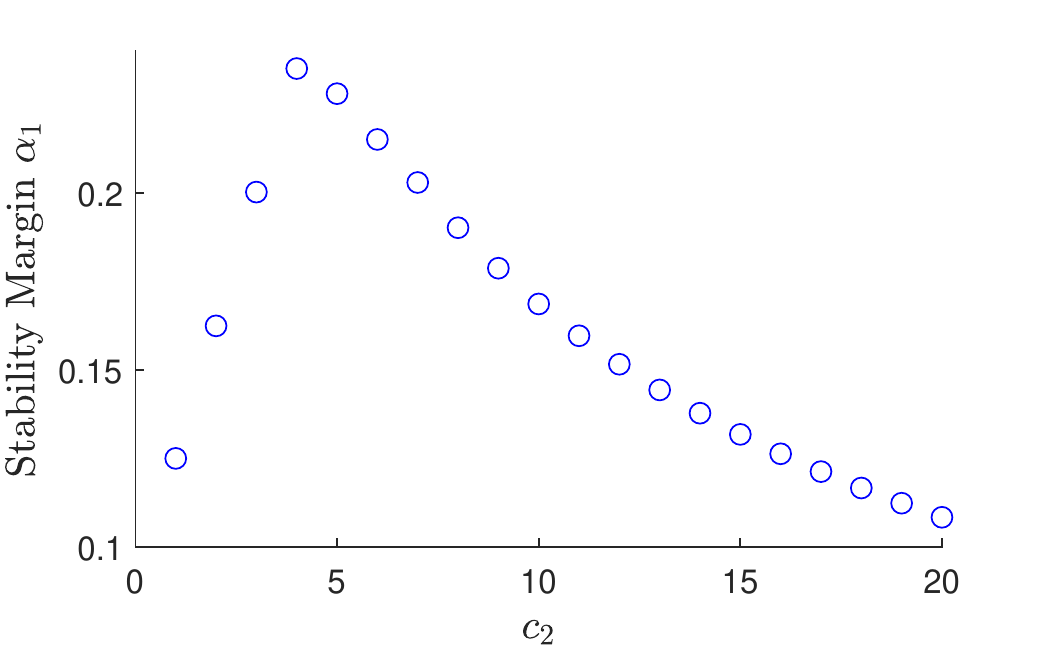} }}
    \qquad
    \subfloat{{\includegraphics[width=5cm]{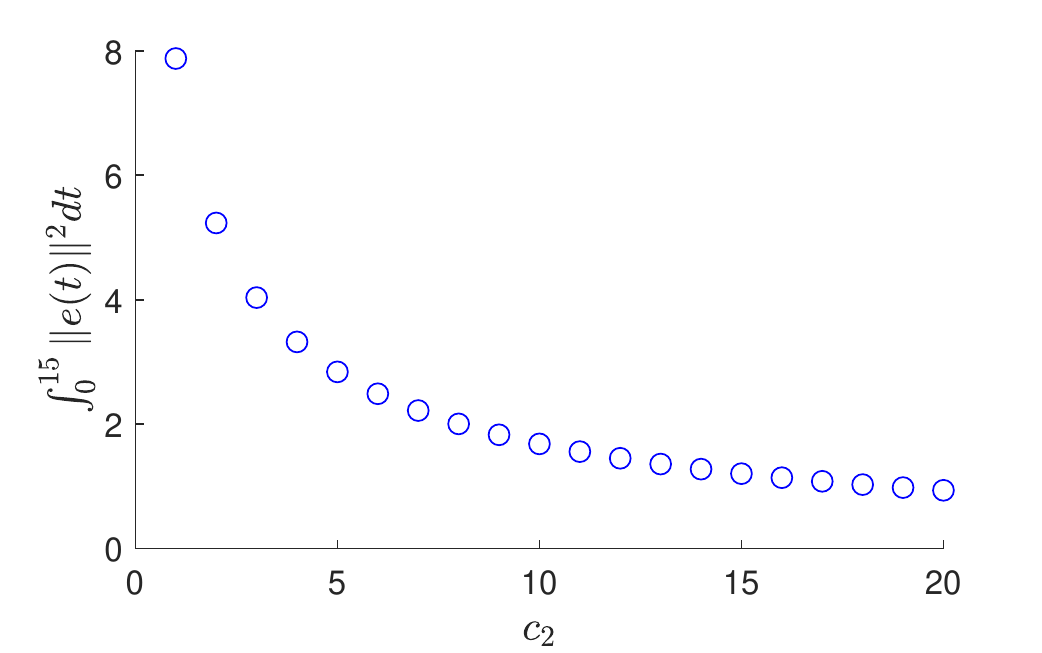} }}
    \caption{The closed-loop stability margin and $\int_0^{15}\norm{e(t)}^2 dt$ for the passive controller with $c_1=2.5$}
    \label{fig:c2}
\end{figure}
\begin{figure}
    \centering
    \subfloat{{\includegraphics[width=5cm]{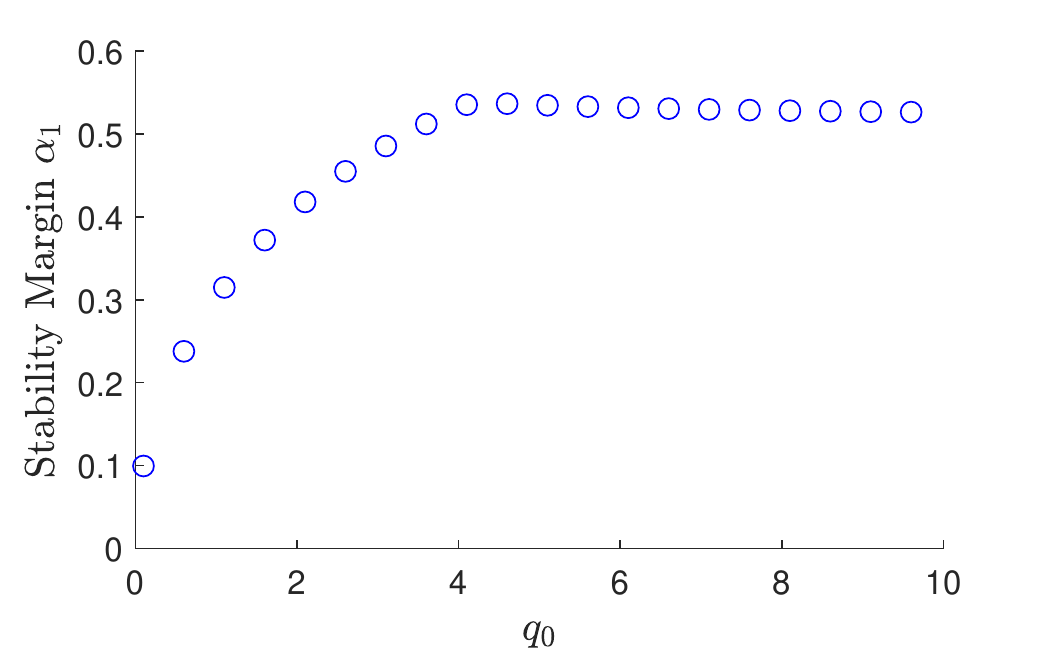} }}
    \qquad
    \subfloat{{\includegraphics[width=5cm]{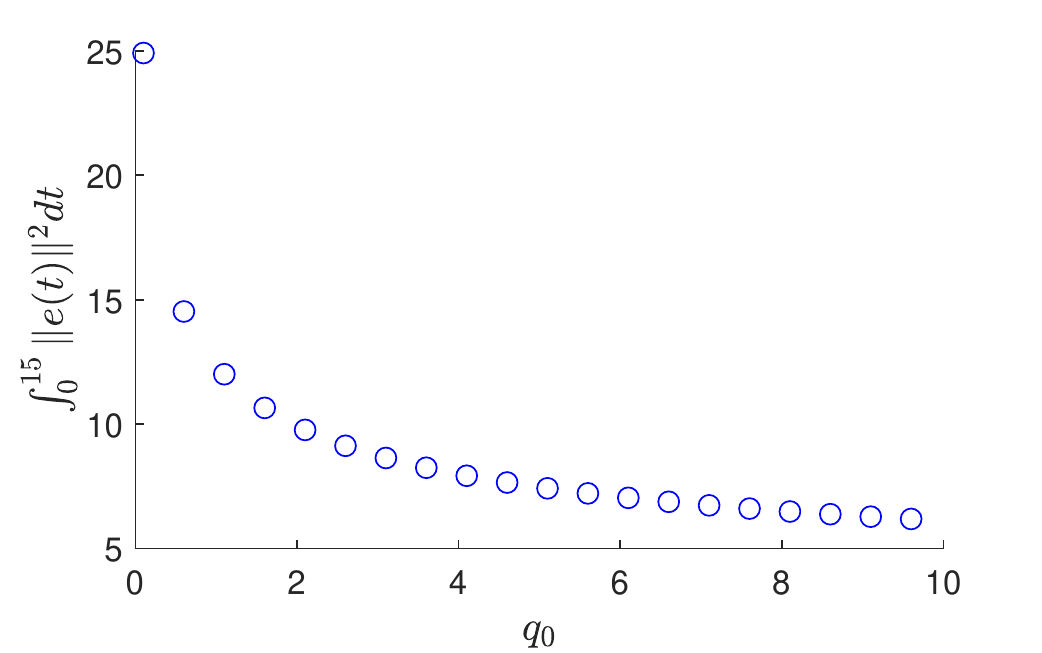} }}
    \caption{The closed-loop stability margin and $\int_0^{15}\norm{e(t)}^2 dt$ for the observer based controller with $R=0.1 I_2$}
    \label{fig:Q}
\end{figure}
\begin{figure}
    \centering
    \subfloat{{\includegraphics[width=5cm]{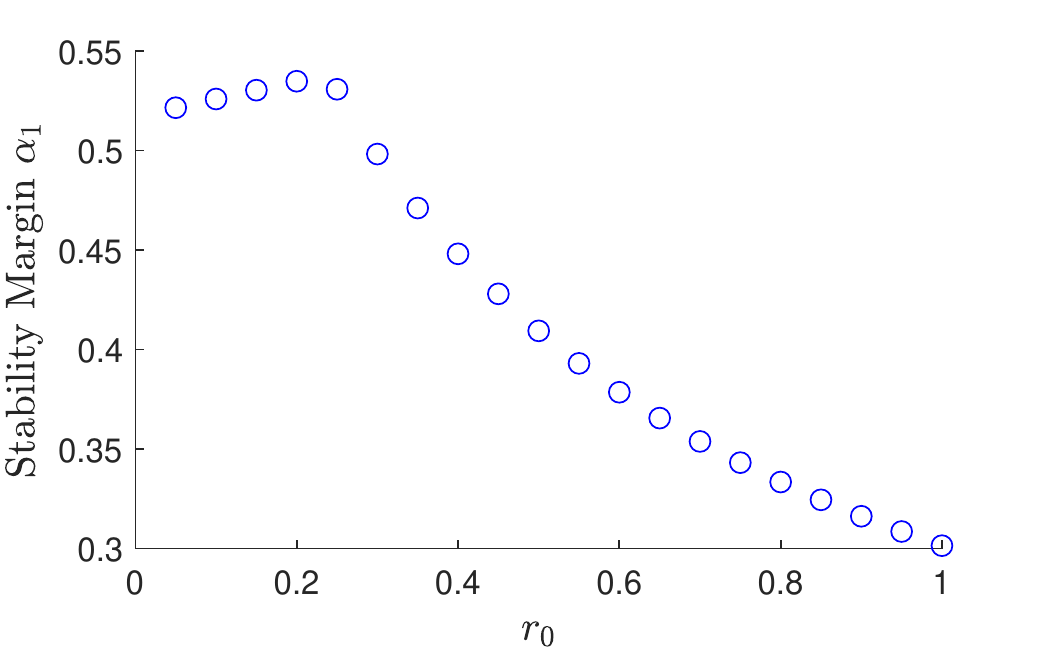} }}
    \qquad
    \subfloat{{\includegraphics[width=5cm]{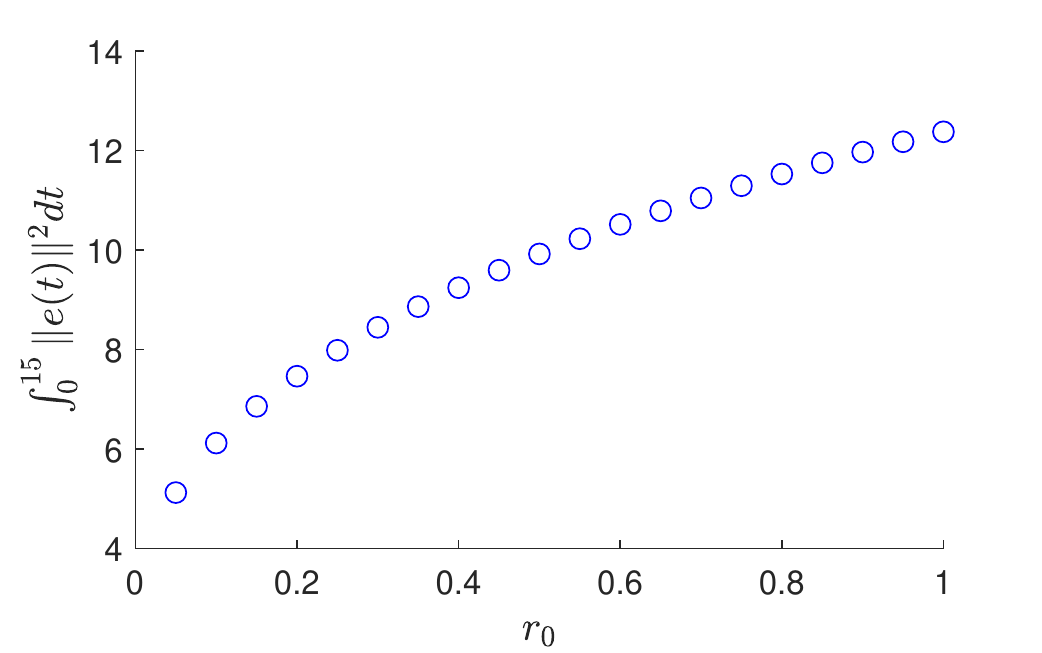} }}
    \caption{The closed-loop stability margin and $\int_0^{15}\norm{e(t)}^2 dt$ for the observer based controller with $Q=10 I_{Z_0}$}
    \label{fig:R}
\end{figure}

\begin{figure}[!ht]
\begin{center}
\includegraphics[scale=0.6]{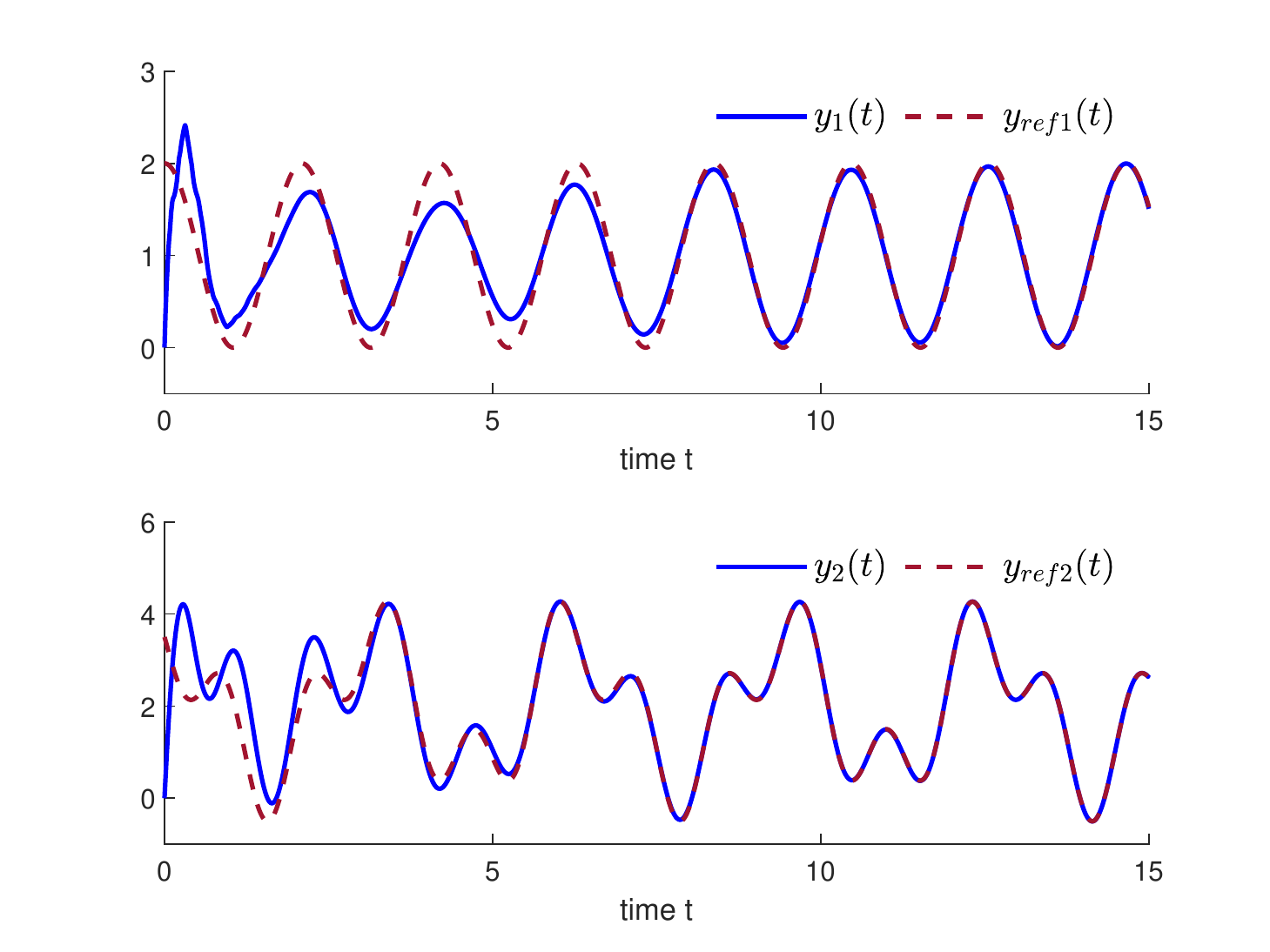}
\end{center}
\caption{Output tracking using a passive controller}
\label{fig.trackinglowgain}
\end{figure}

\begin{figure}[!ht]
\begin{center}
\includegraphics[scale=0.6]{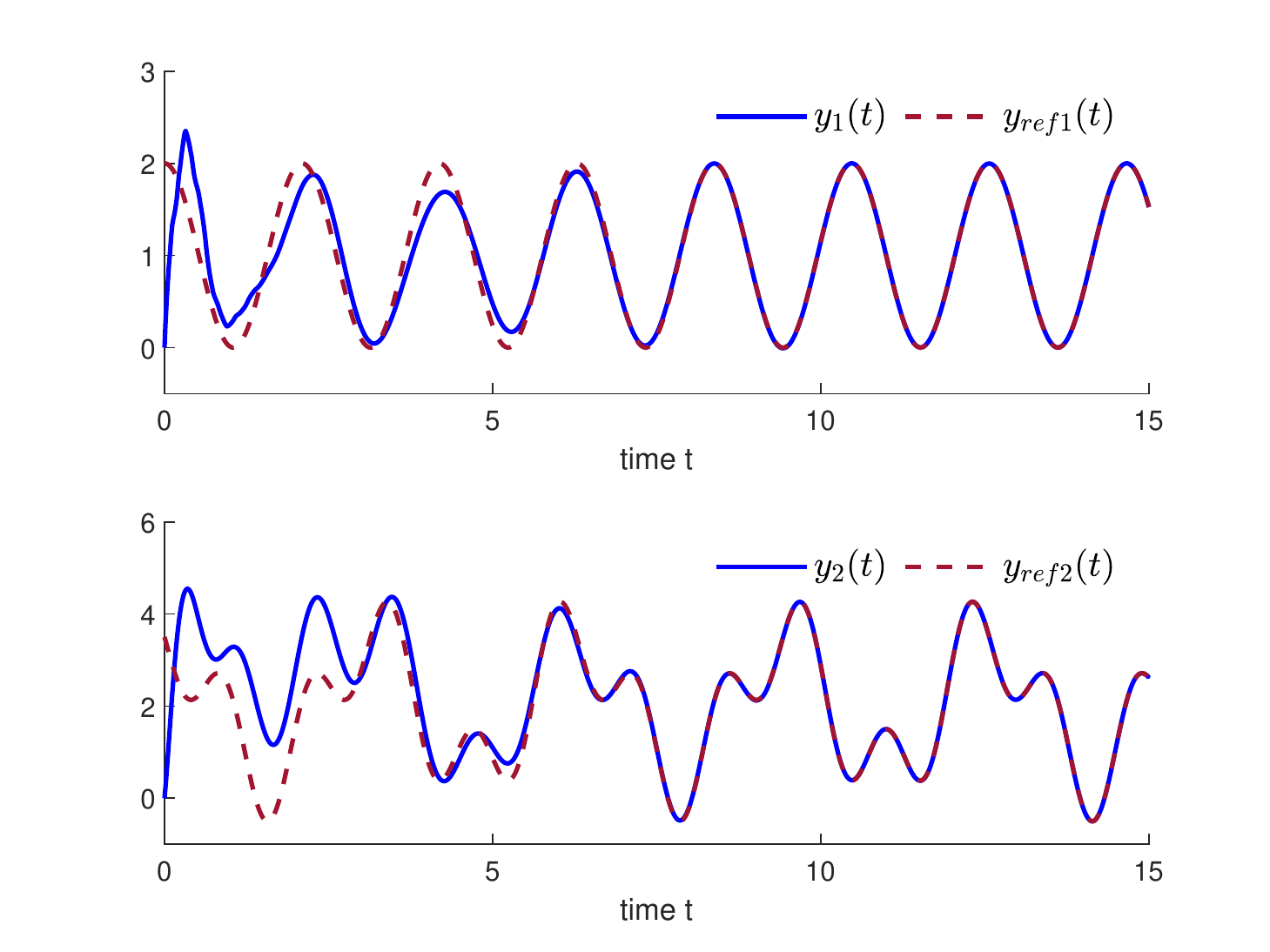}
\end{center}
\caption{Output tracking using an observer based controller}
\label{fig.tracking}
\end{figure}
\begin{figure}[!ht]
\begin{center}
\includegraphics[scale=0.6]{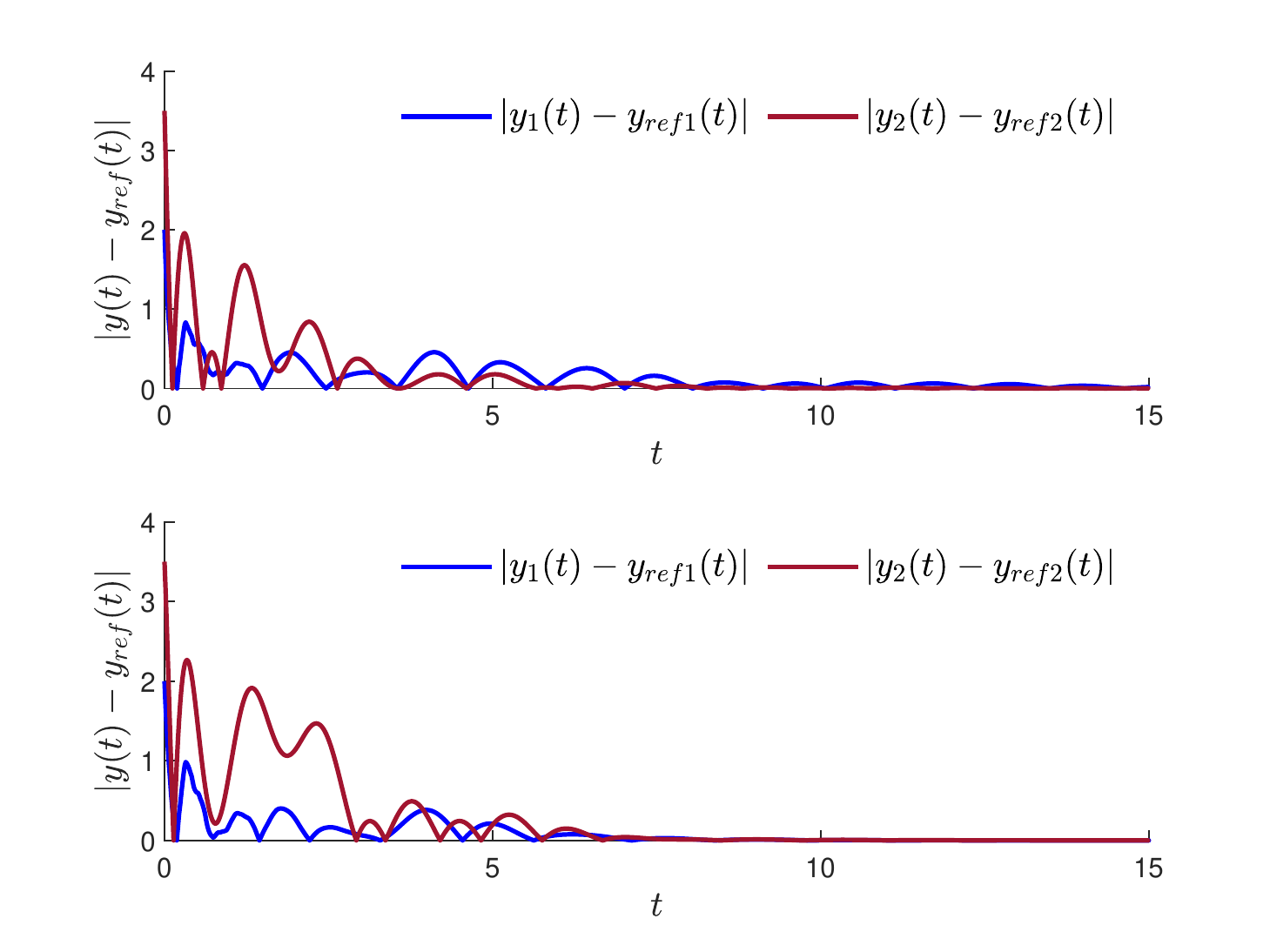}
\end{center}
\caption{Tracking errors for passive(above) and observer based(below) controllers }
\label{fig.trackingerrors}
\end{figure}

 The components of the observer based controller are chosen as in Section \ref{sec.obs-con}.
The matrix $H$ is obtained by solving the system (\ref{eq.sylsol}), where we use the approximations $A^N$ and $C^N$ in place of $A$ and $C$, respectively. The gain matrix $K_1$ is obtained using Matlab lqr function with $Q= q_0 I_{Z_0},\ q_0>0$ and $R=r_0 I_{2}, r_0>0$. To maximize the stability margin, ranges of values of the parameters $q_0$ and $r_0$ were tested. The closed-loop stability margin $\alpha_1$ and $\int_0^{15}{\norm{e(t)}^2 dt}$ for different parameter values are plotted in Figures \ref{fig:Q} and \ref{fig:R}, respectively. It is observed that smaller control gains $r_0$ and larger $q_0$ result in larger closed-loop stability margin and smaller transient errors. By choosing $q_0=10$ and $r_0=0.1$, the output tracking and the tracking errors are depicted in Figures  \ref{fig.tracking} and \ref{fig.trackingerrors}, respectively.

It can be seen from the figures that both controllers achieve tracking of the given reference signals asymptotically and the tracking error decays to zero at an exponential rate. Moreover, we can see that the observer based controller can achieve larger closed-loop stability margin and therefore the asymptotic error convergence for the observer based controller is faster than that for the passive controller.  On the other hand, it is noted that even though the passive controller is a finite-dimensional controller and also the controller requires no information about the satellite system apart from passivity, it still achieves comparable performance to the infinite-dimensional observer based controller.

\section{Conclusion}\label{sec.con}
We investigated robust output tracking problem of a flexible satellite composed of two identical flexible solar panels and a center rigid body. A detailed proof of exponential stability of the model was presented. We constructed two robust controllers for the robust output tracking of the satellite model.
Moreover, simulation results showing the performances of the controllers were presented.

\end{document}